\documentclass{article}
\usepackage{eurosym}
\usepackage{amsmath}
\usepackage{amssymb}
\usepackage{amstext}
\usepackage{amsfonts}
\usepackage{colortbl}

\newtheorem{theorem}{Theorem}[section]

\newtheorem{example}[theorem]{Example}

\newtheorem{lemma}[theorem]{Lemma}

\newtheorem{proposition}[theorem]{Proposition}
\newtheorem{remark}[theorem]{Remark}

\newenvironment{proof}[1][Proof]{\noindent\textbf{#1.} }{\ \rule{0.5em}{0.5em}}
\numberwithin{equation}{section}

\begin{document}

\title{Global boundedness of weak solutions with finite energy\ to a general
class of Dirichlet problems}
\author{Giovanni Cupini$^{1}$, Paolo Marcellini$^{2}$\thanks{%
Corresponding author.} \\
%EndAName
$\quad $ \\
{\normalsize $^{1}$Dipartimento di Matematica, Universit\`a di Bologna }\\
{\normalsize Piazza di Porta S. Donato 5, 40126 - Bologna, Italy }\\
{\normalsize giovanni.cupini@unibo.it }\\
{\normalsize $^{2}$Dipartimento di Matematica e Informatica ``U. Dini'',
Universit\`a di Firenze}\\
{\normalsize Viale Morgagni 67/A, 50134 - Firenze, Italy}\\
{\normalsize paolo.marcellini@unifi.it}}
\date{}
\maketitle

\begin{abstract}
As explained in detail in the prologue to this manuscript, boundedness of
weak solutions for general classes of elliptic equations in divergence form
is a classic tool for achieving higher regularity. We propose here some
global boundedness results under general assumptions that can be applied to
several cases studied in the recent and extensive literature on partial
differential equations \textit{under general growth}. In particular, we
propose the class of \textit{weak solutions with finite energy} in which to
search for solutions and in which regularity can be studied and achieved. We
emphasize that we are not limited to minimizers of certain integral
functionals, as often considered recently in this context of general growth,
but to the broader class of weak solutions to Dirichlet problems for general
nonlinear elliptic equations in divergence form.
\end{abstract}

\bigskip



\bigskip

\section{Prologue}

After Guido Stampacchia \cite{Stampacchia 1965}, who in the 60s focused
Ennio De Giorgi's techniques on the study of second-order linear elliptic
differential equations in divergence form, the book by
Ladyzhenskaya-Uraltseva \cite{Ladyzhenskaya-Uraltseva 1968}, published in
1968, and the monograph by Gilbarg-Trudinger \cite{Gilbarg-Trudinger 1977},
dated 1977, are well known classical references for nonlinear partial
differential equations of elliptic type. We start from them to better
understand the regularity properties and problems that we are going to
discuss in this manuscript.

Theorem 7.1 at page 286 of \cite{Ladyzhenskaya-Uraltseva 1968} is a global
boundedness result for weak solutions to the following class of elliptic
equations in divergence form 
\begin{equation}
\operatorname{div}a\left( x,u,Du\right) =b\left( x,u,Du\right) ,
\label{differential equation in the prologue}
\end{equation}%
where $x\in \Omega ,$ bounded open set in $\mathbb{R}^{n}$ for some $n\geq 2$%
, $u:\Omega \rightarrow \mathbb{R}$ and $Du:\Omega \rightarrow \mathbb{R}%
^{n} $ its gradient.

\begin{theorem}[{Ladyzhenskaya-Uraltseva \protect\cite[Theor. 7.1, Chapter 4]%
{Ladyzhenskaya-Uraltseva 1968}}]
\label{Ladyzhenskaya-Uraltseva Theorem}Assume the coercivity condition for
the vector field $a=a\left( x,u,\xi \right) $, $a:\Omega \times \mathbb{R}%
\times \mathbb{R}^{n}\rightarrow \mathbb{R}^{n}$, 
\begin{equation}
\left( a\left( x,u,\xi \right) ,\xi \right) \geq c_{0}\left( \left\vert
u\right\vert \right) \,\left\vert \xi \right\vert ^{p}-\left( 1+\left\vert
u\right\vert ^{\theta _{1}}\right) c_{1}\left( x\right)
\label{coercivity condition Lad-Ural}
\end{equation}%
and the unilateral growth for $b=b\left( x,u,\xi \right) $, $b:\Omega \times 
\mathbb{R}\times \mathbb{R}^{n}\rightarrow \mathbb{R}\,$, 
\begin{equation}
{\operatorname{sgn}}\left( {u}\right) {b\left( x,u,\xi \right) \geq -\,c_{2}\left(
x\right) \left( 1+\left\vert u\right\vert ^{\theta _{2}}\right) \left\vert
\xi \right\vert ^{n-\varepsilon }-\left( 1+\left\vert u\right\vert ^{\theta
_{2}}\right) c_{3}(x)\,,}  \label{growth condition for b Lad-Ural}
\end{equation}%
for some nonnegative $c_{i}$, $i=0,1,2,3$, (in particular $c_{0}\left(
\left\vert u\right\vert \right) \geq const>0$) and some other restrictions
on the parameters (as described at page 286 of \cite{Ladyzhenskaya-Uraltseva
1968}). Then every weak solution to the differential equation (\ref%
{differential equation in the prologue}) in the Sobolev class $W^{1,p}\left(
\Omega \right) \cap L^{p^{\ast }}\left( \Omega \right) $, which is bounded
at the boundary $\partial \Omega $, is also bounded in all $\overline{\Omega 
}$.
\end{theorem}

A similar result was obtained by Gilbarg-Trudinger \cite{Gilbarg-Trudinger
1977} in Chapter 10 of their celebrated book. The authors named this result 
\textit{maximum principle for divergence form operators} (their quasilinear
elliptic operators are defined in formula (10.5) of \cite{Gilbarg-Trudinger
1977}). The boundedness result was formulated as \textit{an a-priori estimate%
} for sub-solutions (and of course for solutions too).

\begin{theorem}[{Gilbarg-Trudinger \protect\cite[Theorem 10.9]%
{Gilbarg-Trudinger 1977}}]
\label{Gilbarg-Trudinger Theorem}Let $u\in C^{0}\left( \overline{\Omega }%
\right) \cap C^{1}\left( \Omega \right) $ be a weak solution to (\ref%
{differential equation in the prologue}). Suppose that the vector field $%
a\left( x,u,\xi \right) $ satisfies the coercivity condition 
\begin{equation}
\left( a\left( x,u,\xi \right) ,\xi \right) \geq c_{0}\,\left( \left\vert
\xi \right\vert ^{p}-\left\vert u\right\vert ^{p}-1\right)
\label{coercivity condition Gil-Trud}
\end{equation}%
and that $b\left( x,u,\xi \right) $ satisfies the unilateral growth
condition 
\begin{equation}
{\operatorname{sgn}(u)b\left( x,u,\xi \right) \geq }\left\{ 
\begin{array}{l}
{-\,c_{1}}\,\left( \left\vert \xi \right\vert ^{p-1}+\left\vert u\right\vert
^{p-1}+1\right) \,,\;\;\;\;\text{if}\;\;p>1 \\ 
{-\,c_{2}}\,,\;\;\;\;\text{if}\;\;p=1%
\end{array}%
\right. {\,,}  \label{growth condition for b Gil-Trud}
\end{equation}%
for some positive constants $c_{i}$, $i=0,1,2$. Then the maximum of $u$ in $%
\overline{\Omega }$ can be estimated by the maximum of $u$ on the boundary $%
\partial \Omega $, the dimension $n$, the measure $\left\vert \Omega
\right\vert $ and the constants $c_{i}$, $i=0,1,2$.
\end{theorem}

Unfortunately there is a problem in these results if the differential
operators involved are not either the standard Laplacian, or the $p-$%
Laplacian, or similar operators satisfying the so-called \textit{natural
growth conditions}, such as for instance in (\ref{growth conditions by
Serrin}) below. In fact, in the Ladyzhenskaya-Uraltseva Theorem \ref%
{Ladyzhenskaya-Uraltseva Theorem}, a weak solution in the Sobolev class $%
W^{1,p}\left( \Omega \right) \cap L^{p^{\ast }}\left( \Omega \right) $ not
necessarily makes finite the pairing 
\begin{equation}
\int_{\Omega }\left( a\left( x,u\left( x\right) ,Du\left( x\right) \right)
,\varphi \left( x\right) \right) \,dx  \label{pairing in the prologue}
\end{equation}%
against a test function $\varphi \in W_{0}^{1,p}\left( \Omega \right) \cap
L^{p^{\ast }}\left( \Omega \right) $, unless the operator has a $p-$growth
not only from below, but from above too. Necessarily this pairing (\ref%
{pairing in the prologue}) must be finite to respect the notion of weak
solution. Similarly, in Theorem \ref{Gilbarg-Trudinger Theorem} by
Gilbarg-Trudinger, a-priori a weak solution $u\in C^{1}\left( \Omega \right) 
$ makes finite the above pairing (\ref{pairing in the prologue}) for any
test function $\varphi \in C_{0}^{1}\left( \Omega \right) $ with compact
support in $\Omega $; however we do not have uniform estimates. The result
makes not easy, sometimes impossible, to proceed from the a-priori estimate
to the final boundedness, again if the differential operator is not of the
standard $p-$Laplacian type, if it satisfies some more general non-standard
growth conditions.

This fact has been clearly pointed out also in the more recent monograph by
Pucci-Serrin \cite{Pucci-Serrin 2007}. In fact, previously, James Serrin
himself already noted in 1964 (see the quasilinear equation (1) and the
coercivity and growth condition (2) in \cite{Serrin 1964}) that it is
necessary to impose a growth condition on the operator in order to make the
pairing well defined. He proposed in \cite[(2) and (6)]{Serrin 1964} what
nowadays is considered the \textit{standard growth} 
\begin{equation}
\left\{ 
\begin{array}{l}
\left( a\left( x,u,\xi \right) ,\xi \right) \geq c_{0}\,\left( \left\vert
\xi \right\vert ^{p}-\left\vert u\right\vert ^{p}-1\right) \\ 
\left\vert a\left( x,u,\xi \right) \right\vert \leq c_{1}\,\big( \left\vert
\xi \right\vert ^{p-1}+\left\vert u\right\vert ^{p-1}+1\big) \\ 
\left\vert b\left( x,u,\xi \right) \right\vert \leq c_{2}\,\big( \left\vert
\xi \right\vert ^{p-1}+\left\vert u\right\vert ^{p-1}+1\big)%
\end{array}%
\right.  \label{growth conditions by Serrin}
\end{equation}%
for some $p>1$ and some positive constants $c_{i}$, $i=0,1,2$. We can easily
test that (\ref{growth conditions by Serrin})$_{2}$ and (\ref{growth
conditions by Serrin})$_{3}$ respectively imply that $\left\vert a\left(
x,u,Du\right) \right\vert \in L_{\mathrm{loc}}^{p^{\prime }}\left( \Omega
\right) $ and $b\left( x,u,Du\right) \in L_{\mathrm{loc}}^{p^{\prime
}}\left( \Omega \right) $ (here as usual $\frac{1}{p^{\prime }}+\frac{1}{p}%
=1 $) whenever $u\in W_{\mathrm{loc}}^{1,p}\left( \Omega \right) $. From the
Appendix of the book by Pucci-Serrin \cite[page 54]{Pucci-Serrin 2007} we
read: \textit{``The delicacy of the structure can be emphasized by observing
first that Gilbarg and Trudinger \cite{Gilbarg-Trudinger 1977} define weak
solutions exactly as we do here (see equation (8.30) in \cite%
{Gilbarg-Trudinger 1977}), while in their following Theorem 8.15 (for the
case of linear equations) they consider solutions in }$W^{1,2}(%
%TCIMACRO{\U{3a9} }%
%BeginExpansion
\Omega
%EndExpansion
)$\textit{, these being 2-regular by linearity and so legitimate in forming
test functions. On the other hand, for Lemma 10.8 in \cite[page 273]%
{Gilbarg-Trudinger 1977} their solution is assumed to be in }$C^{1}(%
%TCIMACRO{\U{3a9} }%
%BeginExpansion
\Omega
%EndExpansion
)$\textit{, so one then must have }$\left\vert a(\text{\textperiodcentered }%
\ ,u,Du)\right\vert \in L_{\mathrm{loc}}^{1}(%
%TCIMACRO{\U{3a9} }%
%BeginExpansion
\Omega
%EndExpansion
)$\textit{\ in order to use the theory of weak solutions. \ldots\ Even here,
however, one must also deal with their later statement that solutions can be
allowed in the space }$W^{1,p}(%
%TCIMACRO{\U{3a9} }%
%BeginExpansion
\Omega
%EndExpansion
)$\textit{, see \cite[page 277]{Gilbarg-Trudinger 1977}. This in turn
requires the }$p-$\textit{regularity condition }$\left\vert a(\text{%
\textperiodcentered }\ ,u,Du)\right\vert \in L_{\mathrm{loc}}^{p^{\prime }}(%
%TCIMACRO{\U{3a9} }%
%BeginExpansion
\Omega
%EndExpansion
)$\textit{, a condition which is not indicated in \cite{Gilbarg-Trudinger
1977}. Of course, this begs the question, under what conditions can one in
fact obtain }$\left\vert a(\text{\textperiodcentered }\ ,u,Du)\right\vert
\in L_{\mathrm{loc}}^{p^{\prime }}(%
%TCIMACRO{\U{3a9} }%
%BeginExpansion
\Omega
%EndExpansion
)$\textit{\ when }$u\in W_{\mathrm{loc}}^{1,p}(%
%TCIMACRO{\U{3a9} }%
%BeginExpansion
\Omega
%EndExpansion
)$\textit{? The simplest (though not the only) answer is in (\ref{growth
conditions by Serrin})}$_{2}$\textit{\thinspace .''}

For completeness we mention Theorems 6.1.1 and 6.1.2 by Pucci-Serrin \cite[%
page 128]{Pucci-Serrin 2007}, who considered, in their terminology (see the
footnote at page 52 of \cite{Pucci-Serrin 2007}), $p-$\textit{regular} weak
solutions to (\ref{differential equation in the prologue}); i.e., solutions
in $W_{\mathrm{loc}}^{1,p}(%
%TCIMACRO{\U{3a9} }%
%BeginExpansion
\Omega
%EndExpansion
)$ under the coercivity and growth conditions (\ref{growth conditions by
Serrin}). \ They proved the global boundedness of solutions (and
subsolutions too). In their Theorem 6.1.2 the constants in (\ref{growth
conditions by Serrin}) can be replaced by functions in some Lebesgue spaces.

Other \textit{global boundedness results} can be found in the literature,
related to elliptic equations in divergence form as in (\ref{differential
equation in the prologue}). Most of them deal with Dirichlet problems for
the $p-$Laplacian and a zero boundary datum. We refer to Guedda-Veron \cite%
{Guedda-Veron 1989} in 1989, the article by Egnell \cite[see the Appendix]%
{Egnell 1998} dated 1998; Talenti \cite{Talenti 1979} in 1979 and Cianchi 
\cite{Cianchi 1997} in 1997, with fine results based on rearrangements. For
more recent Dirichlet problems of $p-$Laplacian-type and other boundary
conditions we refer for instance to Marino-Winkert \cite{Marino-Winkert 2019}%
.

Since 1989 more general assumptions (in particular those named $p,q-$\textit{%
growth conditions}) have been introduced in \cite{Marcellini 1989}-\cite%
{Marcellini JOTA 1996}, in order to obtain $C_{\mathrm{loc}}^{0,1}\left(
\Omega \right) =W_{\mathrm{loc}}^{1,\infty }\left( \Omega \right) $ and $%
C^{1,\alpha }\left( \Omega \right) $ regularity to weak solutions of general
classes of differential equations as in (\ref{differential equation in the
prologue}). We mean, for example, related to these model cases 
\begin{equation}
\operatorname{div}a\left( Du\right) :=\sum_{i=1}^{n}\tfrac{\partial }{\partial x_{i}}%
\left( \left\vert u_{x_{i}}\right\vert ^{p_{i}-2}u_{x_{i}}\right)
,\;\;\;a\left( \xi \right) =\left( a_{i}\left( \xi \right) \right)
_{1=1,2,\ldots ,n}:=\left( \left\vert \xi _{i}\right\vert ^{p_{i}-2}\xi
_{i}\right) ;  \label{example 1 in prologue}
\end{equation}%
\begin{equation}
\operatorname{div}a\left( Du\right) :=\sum_{i=1}^{n}\tfrac{\partial \left(
\left\vert Du\right\vert ^{p-2}\log \left( 1+\left\vert Du\right\vert
^{2}\right) u_{x_{i}}\right) }{\partial x_{i}},\;\;\;a\left( \xi \right)
:=\left( \left\vert \xi \right\vert ^{p-2}\log \left( 1+\left\vert \xi
\right\vert ^{2}\right) \xi _{i}\right) ;  \label{example 2 in prologue}
\end{equation}%
\begin{equation}
\operatorname{div}a\left( x,Du\right) :=\sum_{i=1}^{n}\tfrac{\partial }{\partial
x_{i}}\left( \left\vert Du\right\vert ^{p\left( x\right) -2}u_{x_{i}}\right)
,\;\;\;\;\;a\left( x,\xi \right) :=\left( \left\vert \xi \right\vert
^{p\left( x\right) -2}\xi _{i}\right) .  \label{example 3 in prologue}
\end{equation}%
Of course coefficients depending on $x$ and $u$ are also allowed. The $p,q-$%
growth conditions ($p\leq q$) are related to the coercivity and growth ($%
c_{i}>0$) 
\begin{equation}
\left\{ 
\begin{array}{l}
\left( a\left( x,u,\xi \right) ,\xi \right) \geq c_{0}\,\left( \left\vert
\xi \right\vert ^{p}-\left\vert u\right\vert ^{p}-1\right) \\ 
\left\vert a\left( x,u,\xi \right) \right\vert \leq c_{1}\,\big( \left\vert
\xi \right\vert ^{q-1}+\left\vert u\right\vert ^{q-1}+1\big)%
\end{array}%
\right.  \label{p,q}
\end{equation}%
and can be read above by posing in (\ref{example 1 in prologue}) $p$ equal
to the minimum of the $\left\{ p_{i}\right\} _{i=1,2,\ldots ,n}$ and $q$ the
maximum; while in (\ref{example 2 in prologue}) we can fix $q:=p+\varepsilon 
$ for an arbitrarily fixed $\varepsilon >0$. In (\ref{example 3 in prologue}%
), finalized to interior regularity we can fix any subset, for instance a
ball $B_{r}$ compactly contained in $\Omega $ and define $p:=\inf \left\{
p\left( x\right) :\;x\in B_{r}\right\} $, $q:=\sup \left\{ p\left( x\right)
:\;x\in B_{r}\right\} $. Note that, if the exponent $p\left( x\right) $ in (%
\ref{example 3 in prologue}) is continuous (as usually considered in
literature), by choosing the radius $r$ of the ball $B_{r}$ sufficiently
small, we can fix in (\ref{example 3 in prologue}) (as well in (\ref{example
2 in prologue})) the ratio $q/p\geq 1$ close to $1$. This remark is useful
when we compare it with the usual (sometimes necessary for \textit{interior
regularity}) assumption\ that $q/p=1+O\left( 1/n\right) $; in the examples (%
\ref{example 2 in prologue}), (\ref{example 3 in prologue}), this assumption
can be satisfied in all dimensions $n\geq 2$ for interior regularity.

\textit{General growth conditions is the context for non-uniform ellipticity}%
. I.e., often elliptic differential operators which satisfy some
non-standard growth conditions are not uniformly elliptic. We recall that
for the differential operator on the left-hand side of (\ref{differential
equation in the prologue}), defined through a vector field $a\left( x,u,\xi
\right) =\left( a_{i}\left( x,u,\xi \right) \right) _{1=1,2,\ldots ,n}$
differentiable with respect to the variable $\xi \in \mathbb{R}^{n}$, 
\textit{ellipticity} can be tested by the inequalities 
\begin{equation}
g_{1}\left( \left\vert \xi \right\vert \right) \left\vert \lambda
\right\vert ^{2}\leq \sum_{i,j=1}^{n}\frac{\partial a^{i}\left( x,u,\xi
\right) }{\partial \xi _{j}}\lambda _{i}\lambda _{j}\leq g_{2}\left(
\left\vert \xi \right\vert \right) \left\vert \lambda \right\vert ^{2},
\label{ellipticity}
\end{equation}%
valid for some positive functions $g_{1},g_{2}:\left[ 0,+\infty \right)
\rightarrow \left[ 0,+\infty \right) $ and for all $\lambda ,\xi \in \mathbb{%
R}^{n}$; for instance, in the $p,q-$growth case we can fix $g_{1}\left(
\left\vert \xi \right\vert \right) =m\left\vert \xi \right\vert ^{p-2}$ and $%
g_{2}\left( \left\vert \xi \right\vert \right) =M\left( 1+\left\vert \xi
\right\vert ^{q-2}\right) $, $0<m\leq M$, or similarly $g_{1}\left(
\left\vert \xi \right\vert \right) =m(1+\left\vert \xi \right\vert
^{2})^{\left( q-2\right) /2}$ and $g_{2}\left( \left\vert \xi \right\vert
\right) =M(1+\left\vert \xi \right\vert ^{2})^{\left( q-2\right) /2}$. While 
\textit{uniformly ellipticity} means that the \textit{modulus of ellipticity}
in (\ref{ellipticity}) is essentially the same from below and from above;
i.e., there exists a constant $c>0$ such that $g_{2}\left( \xi \right) \leq
c\,g_{1}\left( \xi \right) $ for all $\xi \in \mathbb{R}^{n}$; for example,
this happens for the $p-$Laplacian operator.

Non-uniformly elliptic problems recently received a strong impulse with the
relevant researches by Cristiana De Filippis, Giuseppe Rosario Mingione, Jan
Kristensen and others; see in particular \cite{De Filippis JMPA 2022}-\cite%
{De Filippis-Stroffolini 2023}. An interesting relevant case with $p,q-$%
growth conditions has been recently studied by Colombo-Mingione \cite%
{Colombo-Mingione 2015},\cite{Colombo-Mingione 2015 Bounded minimizers},
Baroni-Colombo-Mingione \cite{Baroni-Colombo-Mingione 2018} and
Eleuteri-Marcellini-Mascolo \cite{Eleuteri-Marcellini-Mascolo 2016},\cite%
{Eleuteri-Marcellini-Mascolo 2020}. It is related to the so-called \textit{%
double phase operators} 
\begin{equation}
\left\{ 
\begin{array}{l}
\sum_{i=1}^{n}\tfrac{\partial }{\partial x_{i}}\left( \left\vert
Du\right\vert ^{p-2}u_{x_{i}}\right) +a\left( x\right) \sum_{i=1}^{n}\tfrac{%
\partial }{\partial x_{i}}\left( \left\vert Du\right\vert
^{q-2}u_{x_{i}}\right) \\ 
a\left( x,\xi \right) =\left( a_{i}\left( x,\xi \right) \right) _{i=1,\ldots
,n}:=\left( \left\{ \left\vert \xi \right\vert ^{p-2}+a\left( x\right)
\left\vert \xi \right\vert ^{q-2}\right\} \xi _{i}\right)_{i=1,\ldots ,n}%
\end{array}%
\right. \,,  \label{double phase}
\end{equation}%
where $1<p<q$ and the continuous coefficient $a\left( x\right) $ is
nonnegative in $\Omega $. The operator is a $p-$Laplacian at the points $%
x\in \Omega $ such that $a\left( x\right) =0$, but essentially is a $q-$%
Laplacian where $a\left( x\right) >0$; in this case the added $p-$Laplacian
being a kind of "lower order" term (i.e., lower growth) with respect to the $%
q-$Laplacian. To be mentioned, since it is particularly relevant with
respect to the boundedness results given in this manuscript, the \textit{%
a-priori} boundedness of solutions assumed by Colombo-Mingione in \cite%
{Colombo-Mingione 2015 Bounded minimizers}. Our Theorems \ref{main
boundeness result} and \ref{main boundeness result with epsilon} well apply
to Dirichlet problems with general right-hand sides $b\left( x,u,Du\right) $
and double phase operators as in (\ref{double phase}).

\textit{Local boundedness} of solution to classes of anisotropic elliptic
equations or systems have been investigated by the authors \cite%
{Cupini-Marcellini-Mascolo 2014}--\cite{Cupini-Marcellini-Mascolo regularity
2024}, by DiBenedetto-Gianazza-Vespri \cite{DiBenedetto-Gianazza-Vespri 2016}
and Cupini-Leonetti-Mascolo \cite{Cupini-Leonetti-Mascolo 2017}.\ \textit{%
Local Lipschitz continuity} of weak solutions to nonlinear elliptic
equations and systems under non standard growth conditions started in 1989.
We refer to Marcellini \cite{Marcellini 1989}--\cite{Marcellini 2023},
Esposito-Leonetti-Mingione \cite{Esposito-Leonetti-Mingione 2004}, the
mentioned above articles by Baroni-Colombo-Mingione \cite{Colombo-Mingione
2015},\cite{Colombo-Mingione 2015 Bounded minimizers},\cite%
{Baroni-Colombo-Mingione 2018} and Eleuteri-Marcellini-Mascolo \cite%
{Eleuteri-Marcellini-Mascolo 2016},\cite{Eleuteri-Marcellini-Mascolo 2020},
specific for \textit{double phase operators} as in (\ref{double phase}). For
related results see also \cite{Bella-Schaffner 2020},\cite{Bildhauer-Fuchs
2005},\cite{Boegelein-Dacorogna-Duzaar-Marcellini-Scheven 2020}, \cite%
{Boegelein-Duzaar-Marcellini-Scheven JMPA 2021},\cite{Byun-Oh 2020},\cite%
{Chlebicka-DeFilippis 2019},\cite{Cianchi-Mazya 2014},\cite%
{Cianchi-Schaffner 2024},\cite{Ciani-Skrypnik-Vespri 2023},\cite%
{Cupini-Marcellini-Mascolo-Passarelli 2023},\cite{DeFilippis-Koch-Kristensen
2024},\cite{DeFilippis-Mingione ARMA 2021},\cite%
{Diening-Harjulehto-Hasto-Ruzicka 2011},\cite{DiMarco-Marcellini 2020},\cite%
{Eleuteri-Marcellini-Mascolo-Perrotta 2022},\cite{Giga-Tsubouchi 2022},\cite%
{Gmeineder-Kristensen 2022},\cite{Hirsch-Schaffner 2021},\cite%
{Liskevich-Skrypnik-2009},\cite{Marcellini-Nastasi-Pacchiano Camacho 2025},%
\cite{Nastasi-Pacchiano Camacho 2023},\cite{Ok 2017}.

Studies on \textit{multiplicity of solutions} often require boundedness of
weak solutions up to the boundary; these fields nowadays are so wide, also
in the context of \textit{general growth conditions}, therefore we limit
ourselves by quoting few articles: at least Mih\u{a}ilescu-Pucci-R\u{a}%
dulescu \cite{Mihailescu-Pucci-Radulescu 2008}, Zhang-R\u{a}dulescu \cite%
{Zhang-Radulescu 2018}, Papageorgiou-R\u{a}dulescu-Zhang \cite%
{Papageorgiou-Radulescu-Zhang 2022}, R\u{a}dulescu-Stapenhorst-Winkert \cite%
{Radulescu-Stapenhorst-Winkert 2025} and the references therein. Finally we
emphasize the article by DeFilippis-Mingione \cite{DeFilippis-Mingione
Notices 2025}, recently published in the \textit{Notices of the American
Mathematical Society}.

Some \textit{global boundedness results} have been recently obtained under
general growth conditions, often related to some specific operators, such as
for instance the double phase operator in (\ref{double phase}) and similar
cases. We refer for instance to Winkert-Zacher \cite[Theorem 1.1]%
{Winkert-Zacher 2012} for variable exponents, Perera-Squassina \cite[%
Proposition 2.4]{Perera-Squassina 2018} and Gasi\~{n}ski-Winkert \cite[%
Theorem 3.1]{Gasinski-Winkert 2020} specific for double-phase equations,
Ho-Kim-Winkert-Zhang \cite[Theorem 3.1]{Ho-Kim-Winkert-Zhang 2022} and
Ho-Winkert \cite[Theorem 4.2]{Ho-Winkert 2023} for double phase problems
with variable exponents; see also Section 3 of the review paper by
Papageorgiou-R\u{a}dulescu \cite{Papageorgiou-Radulescu 2024}. The authors'
recent article \cite{Cupini-Marcellini 2025} to our knowledge is the first
attempt to define a large general class of differential operators, and a
corresponding set of assumptions, which produce global boundedness of $%
W^{1,q}\left( \Omega \right) -$\textit{weak solutions} to the associated
Dirichlet problems. One of main differences with respect to this manuscript
is that here we are able to manage a natural Sobolev class included in $%
W^{1,p}\left( \Omega \right) $, in fact intermediate between $W^{1,p}\left(
\Omega \right) $ and $W^{1,q}\left( \Omega \right) $ - \textit{the class of
weak solutions with finite energy} - which is compatible with weak solutions
of Dirichlet problems associated to the differential equation (\ref%
{differential equation in the prologue}) and which not necessarily are
minimizer of an energy integral, and in fact usually in our context these
solutions are not minimizers.

Some more details about the manuscript \cite{Ho-Winkert 2023} published by
Ho-Winkert in 2023: it is one of the most recent and quoted article, about
regularity in this general context. The \textit{global boundedness} obtained
in \cite[Theorem 4.2]{Ho-Winkert 2023}, partly consequence of some embedding
results for Musielak--Orlicz Sobolev spaces, is related to the differential
operator of double-phase type with variable exponents $q\left( x\right) \geq
p\left( x\right) >1$ 
\begin{equation}
\sum_{i=1}^{n}\tfrac{\partial }{\partial x_{i}}\left\{ \left( \left\vert
Du\right\vert ^{p\left( x\right) -2}+a\left( x\right) \left\vert
Du\right\vert ^{q\left( x\right) -2}\right) u_{x_{i}}\right\} \,.
\label{double phase with variable exponents}
\end{equation}%
One of their conditions (see \cite[assumption $\left( H3\right) $ at page 9]%
{Ho-Winkert 2023}) concerns the exponents $p\left( x\right) ,q\left(
x\right) $ and the coefficient $a\left( x\right) $, which are Lipschitz
continuos functions in $\overline{\Omega }$, such that $1<\min \left\{ 
\tfrac{q\left( x\right) }{p\left( x\right) }:\;x\in \overline{\Omega }%
\right\} $ and $\max \left\{ \tfrac{q\left( x\right) }{p\left( x\right) }%
:\;x\in \overline{\Omega }\right\} <1+\tfrac{1}{n}$. The boundedness theorem
for weak solutions and the embedding results in \cite{Ho-Winkert 2023} for
sure are interesting and well emphasize the delicacy of the subject. On the
contrary, of course these results are conditioned by the \textit{structure
condition} imposed by the \textit{specific differential operator} (\ref%
{double phase with variable exponents}).

In this manuscript we propose an approach to global boundedness regularity
results for weak solutions for \textit{a large class of differential
operators with general growth conditions}, including most of the operators
already considered in literature, for instance all the above examples (\ref%
{example 1 in prologue}),(\ref{example 2 in prologue}),(\ref{example 3 in
prologue}),(\ref{double phase}),(\ref{double phase with variable exponents}%
). In the next Section we formulate in detail all the assumptions and the
global boundedness conclusion. Despite the fact that Theorems \ref{main
boundeness result} and \ref{main boundeness result with epsilon} can be
applied to a broad class of elliptic Dirichlet problems, we emphasize that
our assumptions are no more complicated than those described above, and
sometimes we simplify the context. Among other simplifications, although we
consider general growth for instance of the type $p,q$, we do not require a
bound of the previous type on the ratio $\frac{q}{p}$.

\section{Introduction and statements of main results\label{Section:
Introduction and statements}}

This manuscript is devoted to the study of weak solution with\textit{\
finite energy} to\ Dirichlet problems related to general class of elliptic
equations in divergence form 
\begin{equation}
\operatorname{div}a\left( x,u,Du\right) =b\left( x,u,Du\right) ,\;\;\;\;\;x\in
\Omega \,,  \label{differential equation}
\end{equation}%
where $\Omega $ is a bounded open set in $\mathbb{R}^{n}$, $n\geq 2$, the
vector field $a\left( x,u,\xi \right) :=\left( a^{i}\left( x,u,\xi \right)
\right) _{i=1,\ldots ,n}$ and the right-hand side $b\left( x,u,\xi \right) $
are Carath\'{e}odory maps defined in $\Omega \times \mathbb{R}\times \mathbb{%
R}^{n}$. We give conditions in order to obtain $u\in L^{\infty }\left(
\Omega \right) $; i.e., to obtain the \textit{global boundedness }of the
weak solutions $u$ in the closure $\overline{\Omega }$ of $\Omega $. These
weak solutions $u$ to the differential equation in (\ref{differential
equation}) belong to \textit{the Sobolev class }$W^{1,F}\left( \Omega
\right) $ of functions with \textit{finite energy}, defined in (\ref{the
Sobolev class}) below, allowing the vector field $a$ and the term $b$ in (%
\ref{differential equation}) \textit{to explicitly depend} on $u$, other
than on $x$ and, of course, on the gradient $Du$ of $u:\Omega \subset 
\mathbb{R}^{n}\rightarrow \mathbb{R}$.

The vector field $a\left( x,u,\xi \right) =\left( a^{i}\left( x,u,\xi
\right) \right) _{i=1,\ldots ,n}$ has a \textit{generalized variational
structure}, in the sense that \textit{it is associated to an energy integral}
of the form 
\begin{equation}
F\left( u\right) =\int_{\Omega }f\left( x,u,Du\right) \,dx\,,
\label{energy integral}
\end{equation}%
where $f:\Omega \times \mathbb{R}\times \mathbb{R}^{n}\rightarrow \left[
0,+\infty \right) $ is a nonnegative \textit{Carath\'{e}odory integrand};
i.e., $f=f\left( x,u,\xi \right) $ is measurable with respect to $x\in
\Omega $ and continuous in $\left( u,\xi \right) \in \mathbb{R\times R}^{n}$%
. Moreover $f\left( x,u,\xi \right) $ is \textit{convex} with respect to the
gradient variable $\xi =\left( \xi _{i}\right) _{i=1,\ldots ,n}\in \mathbb{R}%
^{n}$ for a.e. $x\in \Omega $ and all $u\in \mathbb{R}$. Once $\left(
x,u\right) $ are fixed, by convexity the function $f$ is locally Lipschitz
continuos with respect to $\xi \in \mathbb{R}^{n}$ and its gradient $D_{\xi
}f:\Omega \times \mathbb{R}\times \mathbb{R}^{n}\rightarrow \mathbb{R}^{n}$, 
$D_{\xi }f=(f_{\xi _{i}}\left( x,u,\xi \right) )_{i=1,\ldots ,n}$, exists
for almost every $\xi $ in $\mathbb{R}^{n}$. $f:\Omega \times \mathbb{R}%
\times \mathbb{R}^{n}\rightarrow \left[ 0,+\infty \right) $ is a nonnegative 
\textit{Carath\'{e}odory integrand}; i.e., $f=f\left( x,u,\xi \right) $ is
measurable with respect to $x\in \Omega $ and continuous in $\left( u,\xi
\right) \in \mathbb{R\times R}^{n}$. Moreover $f\left( x,u,\xi \right) $ is 
\textit{convex} with respect to the gradient variable $\xi =\left( \xi
_{i}\right) _{i=1,\ldots ,n}\in \mathbb{R}^{n}$ for a.e. $x\in \Omega $ and
all $u\in \mathbb{R}$. Once $\left( x,u\right) $ are fixed, by convexity the
function $f$ is locally Lipschitz continuos with respect to $\xi \in \mathbb{%
R}^{n}$ and its gradient $D_{\xi }f:\Omega \times \mathbb{R}\times \mathbb{R}%
^{n}\rightarrow \mathbb{R}^{n}$, $D_{\xi }f=(f_{\xi _{i}}\left( x,u,\xi
\right) )_{i=1,\ldots ,n}$, exists for almost every $\xi $ in $\mathbb{R}%
^{n} $. We assume that the function $\xi \mapsto f(x,u,\xi )$ and the
gradient-vector-field of its first derivatives $D_{\xi }f\left( x,u,\xi
\right) $ are Carath\'{e}odory maps.

We also assume that $f=f\left( x,u,\xi \right) $ is a $\Delta _{2}-$\textit{%
function with respect to} $\xi \in \mathbb{R}^{n}$; i.e., 
\begin{equation}
f\left( x,u,2\xi \right) \leq M\,f\left( x,u,\xi \right) \,
\label{Delta-two}
\end{equation}%
for a constant $M>1$ and for all $\left( x,u,\xi \right) \in \Omega \times 
\mathbb{R}\times \mathbb{R}^{n}$. For instance, if $f\left( x,u,\xi \right)
:=\left\vert \xi \right\vert ^{p}+a\left( x,u\right) \left\vert \xi
\right\vert ^{q}$ with $a\geq 0$ and $1\leq p\leq q$, then $f\left( x,u,2\xi
\right) \leq 2^{q}\,f\left( x,u,\xi \right) $. When $f\left( \xi \right)
:=\left\vert \xi \right\vert ^{p}\log \left( 1+\left\vert \xi \right\vert
\right) $ then $f\left( 2\xi \right) \leq c\,2^{p}\,f\left( \xi \right) $
for a positive constant $c$. In fact we can choose $c=2=\sup \left\{ g\left(
t\right) :\;t\in \left( 0,+\infty \right) \right\} $, where $g\left(
t\right) :=\frac{\log \left( 1+2t\right) }{\log \left( 1+t\right) }$; the
constant $c$ is finite since a simple computation shows that $%
\lim_{t\rightarrow 0^{+}}g\left( t\right) =2$ and $\lim_{t\rightarrow
+\infty }g\left( t\right) =1$; finally we observe that $c=2$ since $\log
\left( 1+2t\right) <\log \left( 1+2t+t^{2}\right) =2\log \left( 1+t\right) $
for every $t>0$. If $f\left( x,u,\xi \right) :=\left\vert \xi \right\vert
^{p\left( x,u\right) }$ with $p\left( x,u\right) $ positive and bounded
function in $\Omega \times \mathbb{R}$, then $f$ satisfies (\ref{Delta-two})
with constant $M:=2^{q^{+}}$ and $q^{+}:=\sup \left\{ p\left( x,u\right)
:\;\left( x,u\right) \in \Omega \times \mathbb{R}\right\} $. If $f\left(
x,u,\xi \right) :=\displaystyle\sum_{i=1}^{n}\left\vert \xi \right\vert
^{p_{i}\left( x,u\right) }$ with $1\leq p_{i}\left( x,u\right) \leq const$
for $i=1,2,\ldots ,n$, then the $\Delta _{2}-$condition (\ref{Delta-two}) is
also satisfied. Of course combinations of these examples hold too. While $%
f\left( \xi \right) :=\exp \left( \left\vert \xi \right\vert \right) $, or $%
f\left( \xi \right) :=\exp (\left\vert \xi \right\vert ^{2})$, are not $%
\Delta _{2}-$functions; note that these exponential functions do not satisfy
also the conclusion of Lemma \ref{Lemma 1}; precisely they do not satisfy (%
\ref{gradient growth}). These exponential cases give rise to non-uniformly
elliptic equations and need to be treated with appropriate specific
techniques (see Cellina-Staicu \cite{Cellina-Staicu 2018}, and properly in
this context \cite{DeFilippis-Mingione ARMA 2021},\cite{DiMarco-Marcellini
2020},\cite{Marcellini 1993},\cite{Marcellini JOTA 1996},\cite%
{Marcellini-Nastasi-Pacchiano Camacho 2025}).

More precisely we also consider the \textit{reverse }$\Delta _{2}-$\textit{%
condition with respect to} $\xi \in \mathbb{R}^{n}$, in the form $m\,f\left(
x,u,\xi \right) \leq f\left( x,u,2\xi \right) $; that is, together with (\ref%
{Delta-two}), 
\begin{equation}
m\,f\left( x,u,\xi \right) \leq f\left( x,u,2\xi \right) \leq M\,f\left(
x,u,\xi \right) \,,  \label{Delta-two (two sides)}
\end{equation}%
for constants $M\geq m>1$ and for all $\left( x,u,\xi \right) \in \Omega
\times \mathbb{R}\times \mathbb{R}^{n}$ (of course $M,m$ independent of $%
x,u,\xi $). As discussed in Remark \ref{Remark}, the above examples (\textit{%
double phase, variable exponents, anisotropic case}) satisfy both $\Delta
_{2}-$\textit{inequalities in} (\ref{Delta-two (two sides)}).

\textit{What do we mean with vector field }$a\left( x,u,\xi \right) $\textit{%
\ associated to the energy integral (\ref{energy integral})? }It may happen
that $a\left( x,u,\xi \right) $ is obtained as the \textit{first variation}
of (\ref{energy integral}); i.e., under the usual notation $\tfrac{\partial
f\left( x,u,\xi \right) }{\partial \xi _{i}}=f_{\xi _{i}}\left( x,u,\xi
\right) $ 
\begin{equation}
a=\left( a^{i}\right) _{i=1,\ldots ,n}=\left( f_{\xi _{i}}\right)
_{i=1,\ldots ,n}=D_{\xi }f\,.  \label{first variation condition}
\end{equation}%
In this case we say that the vector field $a=\left( a^{i}\right)
_{i=1,\ldots ,n}$ has a \textit{variational structure}. However, in this
paper our \textit{vector field }$a\left( x,u,\xi \right) $ does not
necessarily satisfy (\ref{first variation condition}). We limit ourselves to
compare the vector field $a\left( x,u,\xi \right) =\left( a^{i}\left(
x,u,\xi \right) \right) _{i=1,\ldots ,n}$ to the energy integrand in this
way 
\begin{equation}
\left\{ 
\begin{array}{l}
\left( a\left( x,u,\xi \right) ,\xi \right) \geq c_{1}\,\left( D_{\xi
}f\left( x,u,\xi \right) ,\xi \right) -c_{2}\left\vert u\right\vert
^{p^{\ast } }-b_{1}\left( x\right) \\ 
\left( a\left( x,u,\xi \right) ,\xi \right) \leq c_{3}\,\left( D_{\xi
}f\left( x,u,\xi \right) ,\xi \right) +c_{4}\left\vert u\right\vert
^{p^{\ast }}+b_{2}\left( x\right)%
\end{array}%
\right.  \label{comparison inequalities}
\end{equation}
for some constants $c_{1},c_{3}>0$, $c_{2},c_{4}\geq 0$, for a.e. $x\in
\Omega $ and every $\left( u,\xi \right) \in \mathbb{R}\times \mathbb{R}^{n}$%
. In (\ref{comparison inequalities}) as usual $\left( \cdot ,\cdot \right) $
denotes the scalar product in{\ $\mathbb{R}^{n}$ and} $p^{\ast }$ is the
Sobolev exponent of $p$; i.e. $p^{\ast }$ is $\tfrac{np}{n-p}$ if $p<n$, and
it can be fixed arbitrarily (greater than $p$) if $p\geq n$. The function $%
b_{1}\left( x\right) $ satisfies the summability conditions $b_{1}\in
L^{s_{1}}\left( \Omega \right) $ for some $s_{1}>1$, while $b_{2}\in
L^{1}\left( \Omega \right) $.

We use the energy integral (\ref{energy integral}) to define the Sobolev
class where to look for solutions. Our weak solutions to the differential
elliptic equation (\ref{differential equation}) are not necessarily
minimizers of the energy integral (\ref{energy integral}).

The \textit{coercivity }is expressed in terms of coercivity of $f:\Omega
\times \mathbb{R}\times \mathbb{R}^{n}\rightarrow \left[ 0,+\infty \right) $ 
\begin{equation}
f\left( x,u,\xi \right) \geq c_{5}\,\left\vert \xi \right\vert ^{p}
\label{coercivity condition}
\end{equation}%
for a positive constant $c_{5}\,$and for a.e. $x\in \Omega $ and every $%
\left( u,\xi \right) \in \mathbb{R}\times \mathbb{R}^{n}$. In terms of $u\in 
\mathbb{R}$ we require 
\begin{equation}
f\left( x,v,\xi \right) \leq c_{6}\,f\left( x,u,\xi \right)
\,,\;\;\;\;\;\forall \;u,v\in \mathbb{R}\;:\;\left\vert v\right\vert \leq
\left\vert u\right\vert \,,  \label{request on the u-dependence}
\end{equation}%
for a positive constant $c_{6}$, a.e. $x\in \Omega $ and for all $\xi \in 
\mathbb{R}^{n}$. We note that in most of the examples considered in
literature, such as for instance (\ref{example 1 in prologue})-(\ref{example
3 in prologue}) and (\ref{double phase}),(\ref{double phase with variable
exponents}), the function $f=f\left( x,\xi \right) $ is independent of the
variable $u\in \mathbb{R}$ and thus condition (\ref{request on the
u-dependence}) is trivially satisfied. Other similar cases, with $f\left(
x,u,\xi \right) $ depending on $u$ too, can be easily deduced by simple
perturbations; for instance, if {$f(x,\cdot ,\xi )$ is decreasing in $%
(-\infty ,0)$ and increasing in $(0,+\infty )$, a situation that happens for
instance with }$f:=g\left( x,\left\vert u\right\vert ,\xi \right) $ and {$%
g(x,\cdot ,\xi )$ increasing. For instance, we can refer to the simple model
cases }%
\begin{equation}
F_{1}\left( u\right) =\int_{\Omega }a\left( x,u\right) \left\vert
Du\right\vert ^{p}\,dx\,,\;\;\;\;\;F_{2}\left( u\right) =\int_{\Omega
}b\left( x,\left\vert u\right\vert \right) \left\vert Du\right\vert
^{p}\,dx\,,  \label{simple model cases}
\end{equation}%
for some $p>1$. In the first model in (\ref{simple model cases}) we have $%
f_{1}\left( x,u,\xi \right) :=a\left( x,u\right) \left\vert \xi \right\vert
^{p}$, where $a:\Omega \times \mathbb{R}\rightarrow \left( 0,+\infty \right) 
$, $m\leq a\left( x,u\right) \leq M$ for some positive constants $m,M$. Then 
$a\left( x,v\right) \leq M\leq \frac{M}{m}\,a\left( x,u\right) $ for all $%
u,v\in \mathbb{R}$; thus assumption (\ref{request on the u-dependence})
holds in the form $f_{1}\left( x,v,\xi \right) \leq \frac{M}{m}f_{1}\left(
x,u,\xi \right) $ for all $u,v$ in $\mathbb{R}$. The function $f_{2}\left(
x,u,\xi\right) :=b\left( x,\left\vert u\right\vert \right)|\xi|^p $ in the
second model in (\ref{simple model cases}) has the coefficient $b:\Omega
\times \left[ 0,+\infty \right) \rightarrow \left[ 0,+\infty \right) $ not
necessarily bounded, not necessarily far from zero (it could also be $%
b\left( x,0\right) =0$ for some $x\in \Omega $), however $b\left( x,\cdot
\right) $ is an increasing function with respect to its second variable;
then it is clear that (\ref{request on the u-dependence}) holds for $%
f_{2}\left( x,u,\xi \right) $ too.

The Carath\'{e}odory function $b\left( x,u,\xi \right) $ on the right-hand
side of (\ref{differential equation}) is defined in $\Omega \times \mathbb{R}%
\times \mathbb{R}^{n}$ and it satisfies the \textit{unilateral growth
conditions}{\ }%
\begin{equation}
\left\{ 
\begin{array}{l}
\operatorname{sign}\left( u\right) \,b\left( x,u,\xi \right) \geq -\,c_{7}\,\big(%
f\left( x,u,\xi \right) ^{1-\frac{1}{p^{\ast }}}+\left\vert \xi \right\vert
^{p+\frac{p}{n}-1}+\left\vert u\right\vert ^{p^{\ast }-1}\big)-b_{3}\left(
x\right) \\ 
u\,b\left( x,u,\xi \right) \leq c_{8}\,\big(f\left( x,u,\xi \right)
+\left\vert u\right\vert ^{p^{\ast }}\big)+b_{4}\left( x\right)%
\end{array}%
\right.  \label{growth condition for b}
\end{equation}%
for some positive constants $c_{i}\,$, for a.e. $x\in \Omega $ and every $%
\left( u,\xi \right) \in \mathbb{R}\times \mathbb{R}^{n}$. The summability
of the {nonnegative functions $b_{3},b_{4}$ }is $b_{3}\in L^{\frac{n}{p}%
+\varepsilon }\left( \Omega \right) $ for some $\varepsilon >0$, while $%
b_{4}\in L^{1}\left( \Omega \right) $. If $p\geq n$, due to the
arbitrariness in the choice of $p^{\ast }$, the exponent $\alpha :=1-\frac{1%
}{p^{\ast }}$ reduces to any fixed real number $\alpha \in \left[ 0,1\right) 
${, while }$p^{\ast }$ {and }$p^{\ast }-1$ to{\ fixed real numbers
arbitrarily large.}

In literature often positive solutions are taken under consideration. Thus
the interest of the \textit{unilateral growth conditions} (\ref{growth
condition for b}) can be seen particularly when ${\operatorname{sgn}}\left( {u\cdot b%
}\right) $ is constant, either greater than or equal to zero, or less than
or equal to zero, so that one of the two inequalities in (\ref{growth
condition for b}) is automatically satisfied. In a more general context it
could be useful to require a single growth for the absolute value of $%
b\left( x,u,\xi \right) $. As shown below in Remark \ref%
{r:boundalternativoperb}, a sufficient condition to (\ref{growth condition
for b}), i.e., a \textit{growth condition} for the right-hand side $b\left(
x,u,\xi \right) $ in (\ref{differential equation}), is 
\begin{equation}
\left\vert b\left( x,u,\xi \right) \right\vert \leq c_{9}\,\big(f\left(
x,u,\xi \right) ^{1-\frac{1}{p^{\ast }}}+\left\vert u\right\vert ^{p^{\ast
}-1}\big)+b_{3}\left( x\right)  \label{growth condition for |b|}
\end{equation}%
for a constants $c_{9}\geq 0$, for a.e. $x\in \Omega $ and every $\left(
u,\xi \right) \in \mathbb{R}\times \mathbb{R}^{n}$. As before, the
summability of $b_{3}$ is $b_{3}\in L^{\frac{n}{p}+\varepsilon }\left(
\Omega \right) $ with $\varepsilon >0$. For the reader's convenience, the
explicit analytic expression of the exponent of $f\left( x,u,\xi \right) $
in (\ref{growth condition for b}) is $1-\frac{1}{p^{\ast }}=1-\big(\frac{1}{p%
}-\frac{1}{n}\big)$. That is, an equivalent formulation of (\ref{growth
condition for |b|}), for $p\leq n$, is 
\begin{equation}
\left\vert b\left( x,u,\xi \right) \right\vert \leq c_{9}\,\big(f\left(
x,u,\xi \right) ^{1-(\frac{1}{p}-\frac{1}{n})}+\left\vert u\right\vert
^{p^{\ast }-1}\big)+b_{3}\left( x\right) \,.
\label{growth condition for |b| - equivalent formulation}
\end{equation}%
\bigskip

We consider here the Dirichlet problem 
\begin{equation}
\left\{ 
\begin{array}{l}
\sum_{i=1}^{n}\frac{\partial }{\partial x_{i}}a^{i}\left( x,u,Du\right)
=b\left( x,u,Du\right) ,\;\;\;\;\;x\in \Omega \,, \\ 
u=u_{0}\;\;\;\;\;\text{on}\;\;\partial \Omega \,,%
\end{array}%
\right.  \label{Dirichlet problem}
\end{equation}%
where $u_{0}:\Omega \rightarrow \mathbb{R}$ is a \textit{bounded} boundary
datum; i.e., $u_{0}\in L^{\infty }\left( \Omega \right) $. As already
described, we consider \textit{weak solutions} $u:\Omega \subset \mathbb{R}%
^{n}\rightarrow \mathbb{R}$ to the differential equation (\ref{Dirichlet
problem})$_{1}$ which solve the Dirichlet problem (\ref{Dirichlet problem}).
These solutions $u$ belong to \textit{the Sobolev class }$W^{1,F}\left(
\Omega \right) $ of functions with \textit{finite energy}, defined by 
\begin{equation}
W^{1,F}\left( \Omega \right) :=\left\{ u\in W^{1,1}\left( \Omega \right)
:\;\;F\left( u\right) :=\int_{\Omega }f\left( x,u,Du\right) \,dx\;<+\infty
\right\} .  \label{the Sobolev class}
\end{equation}%
Similarly we denote by $W_{0}^{1,F}\left( \Omega \right) :=W^{1,F}\left(
\Omega \right) \cap W_{0}^{1,p}\left( \Omega \right) $. In order to have the
zero function in these Sobolev classes $W^{1,F}\left( \Omega \right) $ and $%
W_{0}^{1,F}\left( \Omega \right) $, we require for $f$ the natural \textit{%
summability property} 
\begin{equation}
x\mapsto f\left( x,0,0\right) \in L^{1}\left( \Omega \right) \,.
\label{summability property}
\end{equation}

\begin{theorem}
\label{main boundeness result}Let $f:\Omega \times \mathbb{R}\times \mathbb{R%
}^{n}\rightarrow \left[ 0,+\infty \right) $, $f=f\left( x,u,\xi \right) $,
be a \textit{Carath\'{e}odory energy function, convex} with respect to the
gradient variable $\xi \in \mathbb{R}^{n}$ for a.e. $x\in \Omega $ and all $%
u\in \mathbb{R}$,\textit{\ satisfying }the $\Delta _{2}-$condition (\ref%
{Delta-two (two sides)}), the coercivity (\ref{coercivity condition}) for
some $p>1$, and (\ref{request on the u-dependence}),(\ref{summability
property}). Let $a=a\left( x,u,\xi \right) $, $a:\Omega \times \mathbb{R}%
\times \mathbb{R}^{n}\rightarrow \mathbb{R}^{n}$, be a vector field with a
generalized variational structure, in the sense that it is associated to $f$
through the comparison inequalities (\ref{comparison inequalities}).
Finally, let $u\in u_{0}+W_{0}^{1,F}\left( \Omega \right) $ be a weak
solution to the Dirichlet problem (\ref{Dirichlet problem}), with a boundary
datum $u_{0}\in L^{\infty }(\Omega )\cap W^{1,F}\left( \Omega \right) $ and
right-hand side $b=b\left( x,u,\xi \right) $ satisfying either (\ref{growth
condition for b}) or (\ref{growth condition for |b|}). Then $u$ is globally
bounded in $\overline{\Omega }$.
\end{theorem}

\bigskip

The proof of Theorem \ref{main boundeness result} is given in Section \ref%
{Section: Proof}. The global boundedness result of Theorem \ref{main
boundeness result} holds for most of examples of weak solutions to elliptic
problems under general growth conditions considered in the literature. Some
details and examples follow below.

\bigskip

Sometimes a \textit{uniform bound} of the $L^{\infty }\left( \Omega \right)
- $norm of weak solutions in terms of data is useful; for instance this
happens for approximating sequences of weak solutions, when a uniform bound
usually gives rise to a compactness property of the sequence. Essentially
with the same proof of Theorem \ref{main boundeness result} we obtain a
uniform estimate of the $L^{\infty }\left( \Omega \right) -$norm of $u$ in
terms of the data if we slightly restrict the stated assumptions in the
following way.

We can compare the vector field $a\left( x,u,\xi \right) =\left( a^{i}\left(
x,u,\xi \right) \right) _{i=1,\ldots ,n}$ to the energy integrand in this
way 
\begin{equation}
\left\{ 
\begin{array}{l}
\left( a\left( x,u,\xi \right) ,\xi \right) \geq c_{1}\,\left( D_{\xi
}f\left( x,u,\xi \right) ,\xi \right) -c_{2}\left\vert u\right\vert
^{p^{\ast }-\varepsilon }-b_{1}\left( x\right) \\ 
\left( a\left( x,u,\xi \right) ,\xi \right) \leq c_{3}\,\left( D_{\xi
}f\left( x,u,\xi \right) ,\xi \right) +c_{4}\left\vert u\right\vert
^{p^{\ast }}+b_{2}\left( x\right)%
\end{array}%
\right.  \label{comparison inequalities with epsilon}
\end{equation}%
for a fixed positive $\varepsilon \in \mathbb{R}$, for some constants $%
c_{1},c_{3}>0$, $c_{2},c_{4}\geq 0$, for a.e. $x\in \Omega $ and every $%
\left( u,\xi \right) \in \mathbb{R}\times \mathbb{R}^{n}$. Moreover, the
Carath\'{e}odory function $b\left( x,u,\xi \right) $ in right-hand side of (%
\ref{differential equation}) is defined in $\Omega \times \mathbb{R}\times 
\mathbb{R}^{n}$ and it satisfies the \textit{unilateral growth conditions}{\ 
}%
\begin{equation}
\left\{ 
\begin{array}{l}
\operatorname{sign}\left( u\right) \,b\left( x,u,\xi \right) \geq -\,c_{7}\,\big(%
f\left( x,u,\xi \right) ^{1-\frac{1}{p^{\ast }}-\varepsilon }+\left\vert \xi
\right\vert ^{p+\frac{p}{n}-1-\varepsilon }+\left\vert u\right\vert
^{p^{\ast }-1-\varepsilon }\big)-b_{3}\left( x\right) \\ 
u\,b\left( x,u,\xi \right) \leq c_{8}\,\big(f\left( x,u,\xi \right)
+\left\vert u\right\vert ^{p^{\ast }}\big)+b_{4}\left( x\right)%
\end{array}%
\right.  \label{growth condition for b with epsilon}
\end{equation}%
for some $\varepsilon \in \left( 0,1\right) $ and positive constants $%
c_{i}\, $, for a.e. $x\in \Omega $ and every $\left( u,\xi \right) \in 
\mathbb{R}\times \mathbb{R}^{n}$. As before, the summability of the {%
nonnegative functions $b_{3},b_{4}$ }is $b_{3}\in L^{\frac{n}{p}+\varepsilon
}\left( \Omega \right) $, while $b_{4}\in L^{1}\left( \Omega \right) $.
Again, a sufficient condition to (\ref{growth condition for b with epsilon}%
), i.e., a \textit{growth condition} for the right-hand side $b\left(
x,u,\xi \right) $ in (\ref{differential equation}), is 
\begin{equation}
\left\vert b\left( x,u,\xi \right) \right\vert \leq c_{9}\,\big(f\left(
x,u,\xi \right) ^{1-\frac{1}{p^{\ast }}-\varepsilon }+\left\vert
u\right\vert ^{p^{\ast }-1-\varepsilon }\big)+b_{3}\left( x\right)
\label{growth condition for |b| with epsilon}
\end{equation}%
for some $\varepsilon \in \left( 0,1\right) $, a constants $c_{9}\geq 0$,
for a.e. $x\in \Omega $ and every $\left( u,\xi \right) \in \mathbb{R}\times 
\mathbb{R}^{n}$. The summability of $b_{3}$ as before is $b_{3}\in L^{\frac{n%
}{p}+\varepsilon }\left( \Omega \right) $.

Under these conditions we obtain the global boundedness of weak solutions
with a uniform estimate of the $L^{\infty }\left( \Omega \right) -$norm in
terms of the data. In Section \ref{s:dimconepsilon} we provide details on
how to derive the proof of Theorem \ref{main boundeness result with epsilon}
from that of Theorem \ref{main boundeness result}.

\begin{theorem}
\label{main boundeness result with epsilon}Let $f:\Omega \times \mathbb{R}%
\times \mathbb{R}^{n}\rightarrow \left[ 0,+\infty \right) $, $f=f\left(
x,u,\xi \right) $, be a \textit{Carath\'{e}odory energy function, convex}
with respect to the gradient variable $\xi \in \mathbb{R}^{n}$ for a.e. $%
x\in \Omega $ and all $u\in \mathbb{R}$,\textit{\ satisfying }the $\Delta
_{2}$ condition (\ref{Delta-two (two sides)}), the coercivity (\ref%
{coercivity condition}) for some $p>1$, and (\ref{request on the
u-dependence}),(\ref{summability property}). Let $a=a\left( x,u,\xi \right) $%
, $a:\Omega \times \mathbb{R}\times \mathbb{R}^{n}\rightarrow \mathbb{R}^{n}$%
, be a vector field with a generalized variational structure, in the sense
that it is associated to $f$ through the comparison inequalities (\ref%
{comparison inequalities with epsilon}). Let $u\in u_{0}+W_{0}^{1,F}\left(
\Omega \right) $ be a weak solution to the Dirichlet problem (\ref{Dirichlet
problem}), with a boundary datum $u_{0}\in L^{\infty }(\Omega )\cap
W^{1,F}\left( \Omega \right) $ and right-hand side $b=b\left( x,u,\xi
\right) $ satisfying either (\ref{growth condition for b with epsilon}) or (%
\ref{growth condition for |b| with epsilon}). Then $u$ is globally bounded
in $\overline{\Omega }$ and there exists $\gamma \geq 1$ such that 
\begin{equation}
\Vert u\Vert _{L^{\infty }(\Omega )}\leq c\left( 1+\Vert u_{0}\Vert
_{L^{\infty }(\Omega )}\right) \cdot \left( 1+\Vert u\Vert _{L^{p^{\ast
}}(\Omega )}\right) ^{\gamma },
\label{estimate in the main theorem with epsilon}
\end{equation}%
where the constant $c$ depends on the data, it depends on $\varepsilon $
too, but it is independent of $u$.
\end{theorem}

\bigskip

The explicit expression of the exponent $\gamma $ in (\ref{estimate in the
main theorem with epsilon}) is 
\begin{equation}
\gamma :=\frac{p^{\ast }-p}{p^{\ast }-\max \left\{ \tfrac{1}{1-\alpha };\;%
\tfrac{p}{p-r+1};\,\theta ;\;s;\;\tfrac{p^{\ast }}{s_{1}};\;\tfrac{p^{\ast }%
}{s_{3}}+1\right\} }\,,  \label{explicit expression of gamma}
\end{equation}%
where the parameters $\alpha ,r,\theta ,s,s_{1},s_{3}$ are defined in
Section \ref{s:dimconepsilon}.

\bigskip

\begin{remark}
\label{Remark in the Introduction}As described in the previous section, when
compared with the known assumptions in similar contexts, our ellipticity and
growth conditions (\ref{comparison inequalities}),(\ref{coercivity condition}%
),(\ref{growth condition for b}) seem to be more general. As explained in
the previous Section, we consider here a large class of Dirichlet problems
which has the double phase case as an example. In this context me mention an
interesting recent \textit{global boundedness result} due to Ho-Winkert \cite%
[Theorems 4.2 and 5.1]{Ho-Winkert 2023}, specific for the double phase
operator as in (\ref{double phase with variable exponents}), which involves
two exponents $p,q$ under the bounds $1<\frac{q}{p}<1+\frac{1}{n}$. On the
contrary in this manuscript, in the specific double phase case, we do not
require a bound on the ratio $\frac{q}{p}$.
\end{remark}

\begin{example}
\label{Example related to the double phase}We consider weak solutions in the
Sobolev class 
\begin{equation*}
u\in W^{1,1}\left( \Omega \right) :\;\;F\left( u\right) :=\int_{\Omega
}\left\{ \tfrac{1}{p}\left\vert Du\right\vert ^{p\left( x\right) }+\tfrac{1}{%
q}a\left( x\right) \left\vert Du\right\vert ^{q\left( x\right) }\right\}
\,dx\;<+\infty
\end{equation*}%
to the Dirichlet problem for the double phase 
\begin{equation}
\left\{ 
\begin{array}{l}
\sum_{i=1}^{n}\tfrac{\partial }{\partial x_{i}}\left\{ \left( \left\vert
Du\right\vert ^{p\left( x\right) -2}+a\left( x\right) \left\vert
Du\right\vert ^{q\left( x\right) -2}\right) u_{x_{i}}\right\} =b\left(
x,u,Du\right) ,\;\;\;\;\;x\in \Omega \,, \\ 
u=u_{0}\;\;\;\;\;\text{on}\;\;\partial \Omega \,.%
\end{array}%
\right.  \label{Dirichlet problem for the double phase}
\end{equation}%
Here $a,p,q:\Omega \subset \mathbb{R}^{n}\rightarrow \mathbb{R}$ are
measurable \textit{bounded functions,} $a\left( x\right) \geq 0$ and $%
p\left( x\right) ,q\left( x\right) \geq p$ \ a.e. in $\Omega $ for a
constant $p>1$. It is not necessary to assume continuity of $a\left(
x\right) ,p\left( x\right) ,q\left( x\right) $, nor that $p\left( x\right)
\leq q\left( x\right) $ a.e. in $\Omega $. The class of the boundary datum
is $u_{0}\in L^{\infty }(\Omega )\cap W^{1,F}\left( \Omega \right) $; the
right-hand side $b=b\left( x,u,\xi \right) $ takes of one of these
forms:\smallskip \newline
\textit{(i)} $b:=b\left( x\right) \in L^{\frac{n}{p}+\varepsilon }\left(
\Omega \right) \,$;\newline
\textit{(ii)} $b:=b\left( x,u\right) $, with $\left\vert b\left( x,u\right)
\right\vert \leq c(x)\,\left( 1+\left\vert u\right\vert ^{p^{\ast
}-1}\right) $ and $c(x)$ bounded;\newline
\textit{(iii)} $b:=b\left( x,\xi \right) $, with $\left\vert b\left( x,\xi
\right) \right\vert \leq \operatorname{const}\left\{ \left\vert \xi \right\vert
^{p\left( x\right) }+a\left( x\right) \left\vert \xi \right\vert ^{q\left(
x\right) }\right\} ^{1-\frac{1}{p}+\frac{1}{n}}$;\newline
\textit{(iv)} or, finally, as a sum of some or all these cases, as either in
(\ref{growth condition for |b|}) or in (\ref{growth condition for |b| -
equivalent formulation}).

Then, as a consequence of Theorem \ref{main boundeness result}, every weak
solution $u\in u_{0}+W_{0}^{1,F}\left( \Omega \right) $ is globally bounded
in $\overline{\Omega }$. Similar conclusion in \textit{(ii)} if $\left\vert
b\left( x,u\right) \right\vert \leq c(x)\,\left( 1+\left\vert u\right\vert
^{s}\right) $ with a lower summability properties for $c(x)$ in dependence
of the exponent $s<p^{\ast }-1$. If $u$ is a weak solution of (\ref%
{Dirichlet problem for the double phase}) and if, for instance, $ub$ is a 
\textit{nonnegative }in $\overline{\Omega }$, then we can use the unilateral
growth condition for $b\left( x,u,\xi \right) $ in (\ref{growth condition
for b}).
\end{example}

\begin{example}
\label{Example related to the double phase with u dependence}This is a
variation of Example \ref{Example related to the double phase}. With similar
notations we consider weak solutions to the Dirichlet problem for the
generalized double phase 
\begin{equation}
\begin{array}{l}
\sum_{i=1}^{n}\tfrac{\partial }{\partial x_{i}}\left\{ \left( \alpha \left(
x,u\right) \left\vert Du\right\vert ^{p\left( x,\left\vert u\right\vert
\right) -2}+\beta \left( x,u\right) \left\vert Du\right\vert ^{q\left(
x,\left\vert u\right\vert \right) -2}\right) u_{x_{i}}\right\} \\ 
=b\left( x,u,Du\right) ,\;\;\;\forall \;x\in \Omega ,%
\end{array}
\label{Dirichlet problem for generalized double phase}
\end{equation}%
with the boundary condition $u\left( x\right) =u_{0}\left( x\right) $ for
all $x\in \partial \Omega $. Here $\alpha ,\beta :\Omega \times \mathbb{R}%
\subset \mathbb{R}^{n+1}\rightarrow \mathbb{R}$ are \textit{Carath\'{e}odory
bounded functions, such that }$0<c_{1}\leq \alpha \left( x,u\right) \leq
c_{2}$ and $\beta \left( x,u\right) \geq 0$, for a.e. $x\in \Omega $, all $%
u\in \mathbb{R}$ and for some positive constants $c_{1}$,$c_{2}$. Moreover $%
q\left( x,\left\vert u\right\vert \right) \geq p\left( x,\left\vert
u\right\vert \right) \geq p$ for a constant $p>1$ and $p\left( x,\cdot
\right) ,q\left( x,\cdot \right) $ are increasing in $\left[ 0,+\infty
\right) $ for a.e. $x\in \Omega $.

The function $\beta $ \textit{satisfies the analogous of (\ref{request on
the u-dependence}); i.e., }$\beta \left( x,v\right) \leq c\,\beta \left(
x,u\right) $ for all $u,v\in \mathbb{R}$ such that $\left\vert v\right\vert
\leq \left\vert u\right\vert $\textit{. For instance this happens if }$\beta
\left( x,u\right) :=\gamma \left( x\right) \delta \left( u\right) $\textit{,
with } $\gamma \geq 0$ measurable bounded function a.e. in $\Omega $ and $%
\delta $ continuous function in $\mathbb{R}$\textit{\ with one of these
properties:}\newline
\textit{(j) either} $0<c_{3}\leq \delta \left( u\right) \leq c_{4}$ for some
constants $c_{3},c_{4}$ and for all $u\in \mathbb{R}$;\newline
\textit{(jj) or} $\delta :=g\left( \left\vert u\right\vert \right) $ and $g:%
\left[ 0,+\infty \right) \rightarrow \left[ 0,+\infty \right) $ is
increasing.

The boundary datum $u_{0}\in L^{\infty }(\Omega )\cap W^{1,F}\left( \Omega
\right) $, where $F$ is the Sobolev class of functions $u\in W^{1,1}\left(
\Omega \right) $ such as $F\left( u\right) :=\int_{\Omega }f\left(
x,u,Du\right) \,dx\;<+\infty $ and 
\begin{equation}
f\left( x,u,Du\right) :=\tfrac{1}{p}\alpha \left( x,u\right) \left\vert
Du\right\vert ^{p\left( x,\left\vert u\right\vert \right) }+\tfrac{1}{q}%
\beta \left( x,u\right) \left\vert Du\right\vert ^{q\left( x,\left\vert
u\right\vert \right) }.  \label{Energy integrand in the example}
\end{equation}%
Since $\alpha \left( x,v\right) \leq c_{2}\leq \frac{c_{2}}{c_{1}}\alpha
\left( x,u\right) $ and similarly (we limit to condition (j) for $\delta $) $%
\delta \left( v\right) \leq \frac{c_{4}}{c_{3}}\delta \left( u\right) $ for
all $u,v\in \mathbb{R}$, then assumption (\ref{request on the u-dependence}) 
$f\left( x,v,Du\right) \leq c\,f\left( x,u,Du\right) $ for $\left\vert
v\right\vert \leq \left\vert u\right\vert $ is satisfied with $c:=\max
\left\{ c_{2}/c_{1};c_{4}/c_{3}\right\} $. Under (jj), condition (\ref%
{request on the u-dependence}) holds with $c:=\max \left\{
c_{2}/c_{1};1\right\} $.

If the right-hand side $b=b\left( x,u,\xi \right) $ is given as in one of
cases (i),(ii),(iii),\textit{(iv)} of the previous Example \ref{Example
related to the double phase} (of course in (iii) $\left\vert b\left( x,u,\xi
\right) \right\vert $ is bounded by the power of the energy integrand (\ref%
{Energy integrand in the example})) then, by Theorem \ref{main boundeness
result}, every weak solution $u\in u_{0}+W_{0}^{1,F}\left( \Omega \right) $
to the Dirichlet problem (\ref{Dirichlet problem for generalized double
phase}) is globally bounded in $\overline{\Omega }$.
\end{example}

\begin{example}
A further variation of the previous Examples \ref{Example related to the
double phase} and \ref{Example related to the double phase with u dependence}
is a Dirichlet problem of the following type 
\begin{equation}
\left\{ 
\begin{array}{l}
\sum_{i=1}^{n}\tfrac{\partial }{\partial x_{i}}\left\{ \left( \alpha \left(
x,u\right) f_{\xi _{i}}\left( x,u,Du\right) +\beta \left( x,u\right) g_{\xi
_{i}}\left( x,u,Du\right) \right) \right\} =b\left( x,u,Du\right) ,\;x\in
\Omega , \\ 
u=u_{0}\;\;\;\;\;\text{on}\;\;\partial \Omega \,,%
\end{array}%
\right.  \label{Dirichlet problem for a general case}
\end{equation}%
with $\alpha ,\beta :\Omega \times \mathbb{R}\subset \mathbb{R}%
^{n+1}\rightarrow \mathbb{R}$ \textit{Carath\'{e}odory functions as in
Example \ref{Example related to the double phase with u dependence}; i.e., }$%
0<c_{1}\leq \alpha \left( x,u\right) \leq c_{2}$ and $\beta \left(
x,u\right) \geq 0$ a.e. in $\Omega $, as either in \textit{(j) of (jj)}
before. A special simple case happens if $\beta \left( x,u\right) $ is
identically equal to zero and $\alpha \left( x,u\right) $ is identically
equal to one, so that the elliptic differential operator in the left-hand
side of (\ref{Dirichlet problem for a general case}) - in this case with $%
\beta \equiv 0$ and $\alpha \equiv 1$ - reduces to $\sum_{i=1}^{n}\tfrac{%
\partial }{\partial x_{i}}f_{\xi _{i}}\left( x,u,Du\right) =b\left(
x,u,Du\right) $.

Here $D_{\xi }f=\left( f_{\xi _{i}}\right) _{i=1,\ldots ,n}$ and $D_{\xi
}g=\left( g_{\xi _{i}}\right) _{i=1,\ldots ,n}$ respectively are the
gradients, with respect to the variable $\xi \in \mathbb{R}^{n}$, of the 
\textit{Carath\'{e}odory energy} functions $f\left( x,u,\xi \right) $ and $%
g\left( x,u,\xi \right) $, both convex with respect to $\xi \in \mathbb{R}%
^{n}$, satisfying the stated assumptions (\ref{Delta-two (two sides)}),(\ref%
{request on the u-dependence}),(\ref{summability property}) and the
coercivity (\ref{coercivity condition}) for the same exponent $p>1$. This
does not necessarily imply that $f$ and $g$ have the same growth as $%
\left\vert \xi \right\vert \rightarrow +\infty $; for example, it could be
that $g$ behaves like $f$ by a logarithm, or the growth of $g$ as $%
\left\vert \xi \right\vert \rightarrow +\infty $ is a power larger than $%
\left\vert \xi \right\vert ^{p}$, etc.

The right-hand side $b\left( x,u,\xi \right) $ in (\ref{Dirichlet problem
for a general case}) does not necessarily depend on derivatives of $f$ and $%
g $ with respect to $u\in \mathbb{R}$, but it satisfies the growth stated
either in (\ref{growth condition for b}) or in (\ref{growth condition for
|b|}); for instance, in this case (\ref{growth condition for |b|}) reads 
\begin{equation*}
\left\vert b\left( x,u,\xi \right) \right\vert \leq c\,\big(\left( \alpha
f+\beta g\right) ^{1-\frac{1}{p^{\ast }} }+\left\vert u\right\vert ^{p^{\ast
}-1 }\big)+b_{0}\left( x\right)
\end{equation*}%
with $b_{0}\in L^{\frac{n}{p}+\varepsilon }\left( \Omega \right) $ for some $%
\varepsilon>0$, a constant $c>0$, for a.e. $x\in \Omega $ and every $\left(
u,\xi \right) \in \mathbb{R}\times \mathbb{R}^{n}$. Under these conditions,
every weak solution $u\in W^{1,F}\left( \Omega \right) $ to Dirichlet
problem (\ref{Dirichlet problem for a general case}) is globally bounded in $%
\overline{\Omega }$ if the boundary value $u_{0}$ is bounded too; i.e., $%
u_{0}\in L^{\infty }(\Omega )\cap W^{1,F}\left( \Omega \right) $. The\textit{%
\ Sobolev class }$W^{1,F}\left( \Omega \right) $\textit{\ of functions with
finite energy,} defined in (\ref{the Sobolev class}), in the context of this
example is given by the \textit{Sobolev functions} $u\in W^{1,1}\left(
\Omega \right) $ such that 
\begin{equation*}
F\left( u\right) :=\int_{\Omega }\left\{ \alpha \left( x,u\right) f\left(
x,u,Du\right) +\beta \left( x,u\right) g\left( x,u,Du\right) \right\}
\,dx\;<+\infty \,.
\end{equation*}
\end{example}

In the next example we emphasize the possibility to apply Theorem \ref{main
boundeness result} to an elliptic operator which is not the first variation
of an energy integral; i.e., the weak solution of the Dirichlet problem (\ref%
{Dirichlet problem}), independently of the right-hand side $b$, is not
necessarily a minimizer of an energy integral. Example \ref{Example with
weak solution which is not a minimizer} is only a model example, in the
sense that it is possible to exhibit a large class of cases where weak
solutions are not minimizers. The boundedness result obtained in the
particular Example \ref{Example with weak solution which is not a minimizer}
to our knowledge cannot be deduced by any regularity result now in the
literature, when the matrix $\left( a_{ij}\right) _{n\times n}$ is \textit{%
not symmetric}, even in the particular case that $a_{ij}\left( x\right) $
and $a\left( x\right) $ are independent of $u$ and $b$ is identically equal
to zero.

\begin{example}
\label{Example with weak solution which is not a minimizer}The Dirichlet
problem for a differential operator, sum of a linear operator (for instance)
and a variational nonuniformly elliptic nonlinear operator, could have this
analytic expression 
\begin{equation}
\left\{ 
\begin{array}{l}
\sum_{i,j=1}^{n}\tfrac{\partial }{\partial x_{i}}\left( a_{ij}\left(
x\right) u_{x_{j}}\right) +\tfrac{\partial }{\partial x_{1}}\left( a\left(
x,u\right) \left\vert u_{x_{1}}\right\vert ^{q-2}u_{x_{1}}\right) =b\left(
x,u,Du\right) ,\;\;\;\;\;x\in \Omega \,, \\ 
u=u_{0}\;\;\;\;\;\text{on}\;\;\partial \Omega \,,%
\end{array}%
\right.  \label{Dirichlet problem for sum of linear and nonlinear operators}
\end{equation}%
where the square matrix $\left( a_{ij}\right) _{n\times n}$\thinspace , of 
\textit{measurable functions }$a_{ij}=a_{ij}\left( x\right) $, is not
necessarily symmetric. But $\left( a_{ij}\right) _{n\times n}$ is definite
positive, in the sense that $c_{1}\left\vert \lambda \right\vert ^{2}\leq
\sum\limits_{i,j=1}^{n}a_{ij}\left( x\right) \lambda _{i}\lambda _{j}\leq
c_{2}\left\vert \lambda \right\vert ^{2}$ for all $\lambda \in \mathbb{R}%
^{n} $, a.e. $x\in \Omega $. Moreover $a\left( x,u\right) $ is a \textit{%
Carath\'{e}odory function such that }$c_{3}\leq a\left( x,u\right) \leq
c_{4} $. Here all the constants $c_{i}$, $i=1-4$, are positive and $q\geq 2$%
. We chose $u_{x_{1}}$ as an example; in fact we could have fixed some other
indices in the set $\left\{ 1,2,\ldots ,n\right\} $ in order to form a not
empty proper subset; we fixed above the index $\left\{ 1\right\} $. The
associated Sobolev class $W^{1,F}\left( \Omega \right) $ in this case is 
\begin{equation}
u\in W^{1,1}\left( \Omega \right) :\;\;F\left( u\right) :=\int_{\Omega
}\left\{ \tfrac{1}{2}\left\vert Du\right\vert ^{2}+\tfrac{1}{q}a\left(
x,u\right) \left\vert u_{x_{1}}\right\vert ^{q}\right\} \,dx\;<+\infty
\label{Sobolev class in the second example}
\end{equation}%
which is equivalent to the set of Sobolev functions $\left\{ u\in
W^{1,2}\left( \Omega \right) :\;u_{x_{1}}\in L^{q}\left( \Omega \right)
\right\} $. The integral $F\left( u\right) $ is coercive in $W^{1,2}\left(
\Omega \right) $ (i.e., $p=2$) but it is not coercive in $W^{1,q}\left(
\Omega \right) $. If the matrix $\left( a_{ij}\right) _{n\times n}$ is not
symmetric, even if $b\left( x,u,Du\right) $ is identically equal to zero,
the weak solutions of the Dirichlet problem (\ref{Dirichlet problem for sum
of linear and nonlinear operators}) are not necessarily minimizers neither
to $F\left( u\right) $ in (\ref{Sobolev class in the second example}) nor to
the integral $G\left( u\right) $ associated to the quadratic form associated 
$\sum\limits_{i,j=1}^{n}a_{ij}\left( x\right) \lambda _{i}\lambda _{j}$;
i.e., 
\begin{equation}
u\in W^{1,1}\left( \Omega \right) :\;\;G\left( u\right) :=\int_{\Omega }%
\big\{\tfrac{1}{2}\sum\limits_{i,j=1}^{n}a_{ij}\left( x\right)
u_{x_{i}}u_{x_{j}}+\tfrac{1}{q}a\left( x,u\right) \left\vert
u_{x_{1}}\right\vert ^{q}\big\}\,dx\;<+\infty \,.
\label{integral G(u) associated to the quadratic form}
\end{equation}%
The reason is that inside the integral of $G\left( u\right) $ we can
represent equivalently the quadratic form 
\begin{equation*}
\sum\limits_{i,j=1}^{n}a_{ij}\left( x\right)
u_{x_{i}}u_{x_{j}}=\sum\limits_{i,j=1}^{n}\overline{a}_{ij}\left( x\right)
u_{x_{i}}u_{x_{j}}\,,
\end{equation*}%
where $\overline{a}_{ij}=\frac{1}{2}\left( a_{ij}+a_{ji}\right) $ and $%
\left( \overline{a}_{ij}\right) _{n\times n}$ is the symmetric part of the
matrix $\left( a_{ij}\right) _{n\times n}$. Therefore the first Euler's
variation of this integral $G\left( u\right) $ does not reproduce the
differential equation (\ref{Dirichlet problem for sum of linear and
nonlinear operators})$_{1}$ in the general case, when the matrix $\left(
a_{ij}\right) _{n\times n}$ is not symmetric.

Condition (\ref{request on the u-dependence}) holds also if $a\left(
x,u\right) $ explicitly depends on $u$, as explained in Example \ref{Example
related to the double phase with u dependence}. By Theorem \ref{main
boundeness result}, when $b=b\left( x,u,Du\right) $ satisfies either the
growth (\ref{growth condition for b}), or (\ref{growth condition for |b|}),
with $p=2$, then any weak solution to the Dirichlet problem (\ref{Dirichlet
problem for sum of linear and nonlinear operators}), with $u_{0}\in
L^{\infty }(\Omega )\cap W^{1,F}\left( \Omega \right) $, is globally bounded
in $\overline{\Omega }$.
\end{example}

\section{Gradient growth of a convex function\label{A preliminary result}}

Related to the energy integral (\ref{energy integral}), \textit{convexity}
of $f$ plays a role in the\textit{\ growth of its gradient} $Df$, as shown
by the following results, which are inspired by Step 2 of Section 2 in \cite%
{Marcellini 1985} (see also \cite[Lemma 2.3]{Marcellini 2025}).

As before $f:\Omega \times \mathbb{R}\times \mathbb{R}^{n}\rightarrow \left[
0,+\infty \right) $ is a \textit{Carath\'{e}odory function}; i.e., $%
f=f\left( x,u,\xi \right) $ is measurable with respect to $x\in \Omega $ and
continuous in $\left( u,\xi \right) \in \mathbb{R\times R}^{n}$. Moreover $%
f\left( x,u,\xi \right) $ is \textit{convex} with respect to the gradient
variable $\xi =\left( \xi _{i}\right) _{i=1,\ldots ,n}\in \mathbb{R}^{n}$
for a.e. $x\in \Omega $ and all $u\in \mathbb{R}$. In the following Lemma %
\ref{Lemma 1} and Lemma \ref{Lemma 2} the variables $\left( x,u\right) $
only play the role of parameters; therefore here there is not need of the
continuity assumption on $f$ with respect to $u\in \mathbb{R}$. We deal with
convexity of $\xi \rightarrow f\left( x,u,\xi \right) $ and the $\Delta
_{2}- $\textit{conditions} (\ref{Delta-two}), (\ref{Delta-two (two sides)}).
We emphasize that the examples cited (\textit{double phase, variable
exponent, anisotropic case}) satisfy the $\Delta _{2}-$\textit{conditions} (%
\ref{Delta-two}) and (\ref{Delta-two (two sides)}), while \textit{%
exponential cases}, such as for instance either $f\left( \xi \right) :=\exp
\left( \left\vert \xi \right\vert \right) $ or $f\left( \xi \right) :=\exp
(\left\vert \xi \right\vert ^{2})$, do not satisfy these $\Delta _{2}-$%
conditions.

\begin{lemma}
\label{Lemma 1}Let $f\left( x,u,\xi \right) $, $f:\Omega \times \mathbb{R}%
\times \mathbb{R}^{n}\rightarrow \left[ 0,+\infty \right) $, be a $\Delta
_{2}$ convex function with respect to $\xi \in \mathbb{R}^{n}$, as in (\ref%
{Delta-two}) with constant $M$. Then the gradient $D_{\xi }f:\Omega \times 
\mathbb{R}\times \mathbb{R}^{n}\rightarrow \mathbb{R}^{n}$ of $f$, $D_{\xi
}f=\left( f_{\xi _{i}}\left( x,u,\xi \right) \right) _{i=1,\ldots ,n}$,
which exists for almost every $\xi $ in $\mathbb{R}^{n}$ once $\left(
x,u\right) $ are fixed, satisfies the growth condition 
\begin{equation}
\left( D_{\xi }f\left( x,u,\xi \right) ,\xi \right) \leq \left( M-1\right)
f\left( x,u,\xi \right) \,,\;\;\;\;\text{a.e.}\;\left( x,u,\xi \right) \in
\Omega \times \mathbb{R}\times \mathbb{R}^{n}.  \label{gradient growth}
\end{equation}
\end{lemma}

\begin{proof}
As before, $\left( \cdot ,\cdot \right) $ denotes the scalar product in{\ $%
\mathbb{R}^{n}$. }In the proof we do not explicitly denote the dependence of 
$f$ on $\left( x,u\right) $, since these variables remain constant. We use
the convexity inequality for the function $f=f\left( \xi \right) $ 
\begin{equation}
f\left( \xi _{1}\right) \geq f\left( \xi _{0}\right) +\left( D_{\xi }f\left(
\xi _{0}\right) ,\xi _{1}-\xi _{0}\right) \,,  \label{in the Lemma - a}
\end{equation}%
valid for almost every $\xi _{0}$ in $\mathbb{R}^{n}$ where $f$ is
differentiable and for every $\xi _{1}\in \mathbb{R}^{n}$.

With the notation $\xi _{0}:=\xi $ and the choice $\xi _{1}:=2\xi $ we can
write 
\begin{equation*}
f\left( 2\xi \right) \geq f\left( \xi \right) +\left( D_{\xi }f\left( \xi
\right) ,\xi \right) \,
\end{equation*}%
and, by the $\Delta _{2}-$condition in (\ref{Delta-two}), we obtain the
conclusion 
\begin{equation}
\left( D_{\xi }f\left( \xi \right) ,\xi \right) \leq f\left( 2\xi \right)
-f\left( \xi \right) \leq \left( M-1\right) \,f\left( \xi \right) \,.
\label{in the Lemma - b}
\end{equation}
\end{proof}

\bigskip

\begin{remark}
\label{Remark}Let us make a check in a model case $f\left( \xi \right)
:=\left\vert \xi \right\vert ^{p}+a\left( x\right) \left\vert \xi
\right\vert ^{q}$, with $1<p\leq q$ and $a\left( x\right) \geq 0$ for $x\in
\Omega $. An elementary computation for $g\left( \xi \right) :=\left\vert
\xi \right\vert ^{p}$ gives 
\begin{equation*}
g_{\xi _{i}}\left( \xi \right) =p\left\vert \xi \right\vert ^{p-1}\frac{\xi
_{i}}{\left\vert \xi \right\vert }=p\left\vert \xi \right\vert ^{p-2}\xi
_{i},\;\;\;\;g_{\xi _{i}}\left( \xi \right) \xi _{i}=p\left\vert \xi
\right\vert ^{p-2}\left( \xi _{i}\right) ^{2}\,.
\end{equation*}%
We obtain $\left( D_{\xi }g\left( \xi \right) ,\xi \right) =\displaystyle %
\sum_{i=1}^{n}g_{\xi _{i}}\left( \xi \right) \xi _{i}=p\,\left\vert \xi
\right\vert ^{p}=p\,g\left( \xi \right) $\thinspace . Then, similarly for $%
f\left( \xi \right) =\left\vert \xi \right\vert ^{p}+a\left( x\right)
\left\vert \xi \right\vert ^{q}$, 
\begin{equation*}
\left( D_{\xi }f\left( \xi \right) ,\xi \right) =p\,\left\vert \xi
\right\vert ^{p}+a\left( x\right) q\,\left\vert \xi \right\vert ^{q}\leq
q\left( \left\vert \xi \right\vert ^{p}+a\left( x\right) \left\vert \xi
\right\vert ^{q}\right) =q\,f\left( \xi \right)
\end{equation*}%
Let us compare this example with the conclusion (\ref{gradient growth}) of
Lemma \ref{Lemma 1}. Since $f\left( 2\xi \right) =2^{p}\,\left\vert \xi
\right\vert ^{p}+a\left( x\right) 2^{q}\left\vert \xi \right\vert ^{q}\leq
2^{q}\left( \left\vert \xi \right\vert ^{p}+a\left( x\right) \left\vert \xi
\right\vert ^{q}\right) =2^{q}f\left( \xi \right) $, then $M=2^{q}$ in the $%
\Delta _{2}-$condition (\ref{Delta-two}). In this case (\ref{gradient growth}%
) reads 
\begin{equation*}
\left( D_{\xi }f(\xi),\xi \right) =p\,\left\vert \xi \right\vert
^{p}+a\left( x\right) q\,\left\vert \xi \right\vert ^{q}\leq \left(
2^{q}-1\right) \left( \left\vert \xi \right\vert ^{p}+a\left( x\right)
\left\vert \xi \right\vert ^{q}\right) \,,\;\;\;\;\forall \;\xi \in \mathbb{R%
}^{n}.
\end{equation*}%
Dividing by $\left\vert \xi \right\vert ^{p}$ when $p<q$ we obtain $%
p+a\left( x\right) q\left\vert \xi \right\vert ^{q-p}\leq \left(
2^{q}-1\right) \left( 1+a\left( x\right) \left\vert \xi \right\vert
^{q-p}\right) $. This gives $p\leq 2^{q}-1$ and $q\leq 2^{q}-1$; i.e., $%
2^{q}\geq $ $1+q$, which is compatible ($2^{q}=\left( 1+1\right)
^{q}=1+q+... $) and it reduces to an equality when $p=q=1$.

About a similar estimate from below for $\left( D_{\xi }f\left( \xi \right)
,\xi \right) $ in (\ref{gradient growth}), to simplify we can check the
example $g\left( \xi \right) :=\left\vert \xi \right\vert ^{p}$, with $%
\left( D_{\xi }g\left( \xi \right) ,\xi \right) =p\,g\left( \xi \right) $.
In this case $g\left( \xi \right) $ and $\left( D_{\xi }g\left( \xi \right)
,\xi \right) $ are equal each other (up to the multiplicative constant $p$);
thus we have symmetric inequalities (in fact equalities in this simpler
case) both in the left and in the right-hand sides of (\ref{gradient growth
(two-sides)}) below. In the next Lemma \ref{Lemma 2} we show that the
two-sides assumption (\ref{Delta-two (two sides)}) is sufficient to obtain
symmetric inequalities in the left and in the right-hand sides, as stated in
(\ref{gradient growth (two-sides)}).

We conclude with the check of the conclusion (\ref{gradient growth
(two-sides)}) of Lemma \ref{Lemma 2} for this example $g\left( \xi \right)
:=\left\vert \xi \right\vert ^{p}$, $p\geq 1$. In the two sides $\Delta
_{2}- $condition (\ref{Delta-two (two sides)}), we have $m=M=2^{p}$. The
left estimate in (\ref{gradient growth (two-sides)}) $2\left( 1-\tfrac{1}{m}%
\right) g\left( \xi \right) \leq \left( D_{\xi }g\left( \xi \right) ,\xi
\right) $, being $\left( D_{\xi }g\left( \xi \right) ,\xi \right)
=p\,\left\vert \xi \right\vert ^{p}$, is equivalent to $2\left(
1-2^{-p}\right) \leq p$, which is compatible and it reduces to an equality
when $p=1$. In fact, the function $p\mapsto h\left( p\right) :=p+2^{1-p}-2$
has derivative $h^{\prime }\left( p\right) =1-2^{1-p}\log 2<1$ for all $p\in 
\mathbb{R}$. Thus $h\left( p\right) $ is increasing in $\mathbb{R}$ and its
minimum value for $p\in \left[ 1,+\infty \right) $ is $h\left( 1\right) =0$;
therefore $h\left( p\right) =p+2^{1-p}-2\geq 0$ for all $p\in \left[
1,+\infty \right) $.
\end{remark}

\begin{lemma}
\label{Lemma 2}Let $f\left( x,u,\xi \right) $, $f:\Omega \times \mathbb{R}%
\times \mathbb{R}^{n}\rightarrow \left[ 0,+\infty \right) $ be a convex
function with respect to $\xi \in \mathbb{R}^{n}$, satisfying the two sides $%
\Delta _{2}-$condition (\ref{Delta-two (two sides)}) with constants $M\geq
m>1$. Then the gradient $D_{\xi }f:\Omega \times \mathbb{R}\times \mathbb{R}%
^{n}\rightarrow \mathbb{R}^{n}$ of $f$ with respect to the $\xi $ variable, $%
D_{\xi }f=(f_{\xi _{i}}\left( x,u,\xi \right) )_{i=1,\ldots ,n}$ satisfies
the growth conditions 
\begin{equation}
2\left( 1-\tfrac{1}{m}\right) \,f\left( x,u,\xi \right) \leq \left( D_{\xi
}f\left( x,u,\xi \right) ,\xi \right) \leq \left( M-1\right) f\left( x,u,\xi
\right) \,,  \label{gradient growth (two-sides)}
\end{equation}%
for a.e. $\left( x,u,\xi \right) \in \Omega \times \mathbb{R}\times \mathbb{R%
}^{n}$.
\end{lemma}

\begin{proof}
We already obtained in Lemma \ref{Lemma 1} the right-hand side of (\ref%
{gradient growth (two-sides)}). As before we do not explicitly denote the
dependence of $f$ on $\left( x,u\right) $ and we use the convexity
inequality (\ref{in the Lemma - a}) $f\left( \xi _{1}\right) \geq f\left(
\xi _{0}\right) +\left( D_{\xi }f\left( \xi _{0}\right) ,\xi _{1}-\xi
_{0}\right) $. When $\xi _{0}:=\xi $ and $\xi _{1}:=\frac{\xi }{2}$ the
convexity inequality gives 
\begin{equation*}
f\left( \tfrac{\xi }{2}\right) \geq f\left( \xi \right) -\left( D_{\xi
}f\left( \xi \right) ,\tfrac{\xi }{2}\right) \,\,.
\end{equation*}%
We read the condition in the left-hand side of (\ref{Delta-two (two sides)})
in the form $m\,f\left( \tfrac{\xi }{2}\right) \leq f\left( \xi \right) $
and we get 
\begin{equation*}
\tfrac{1}{m}f\left( \xi \right) \geq f\left( \tfrac{\xi }{2}\right) \geq
f\left( \xi \right) -\tfrac{1}{2}\left( D_{\xi }f\left( \xi \right) ,\xi
\right) \,\,.
\end{equation*}%
Therefore $\tfrac{1}{2}\left( D_{\xi }f\left( \xi \right) ,\xi \right) \geq
\left( 1-\tfrac{1}{m}\right) f\left( \xi \right) $, which corresponds to the
left-hand side of the conclusion (\ref{gradient growth (two-sides)}).\bigskip
\end{proof}

\section{Proof of the global boundedness results\label{Section: Proof}}

We prove in this section the global boundedness Theorem \ref{main boundeness
result} for weak solutions $u:\Omega \subset \mathbb{\ R}^{n}\rightarrow 
\mathbb{R}$, $n\geq 2$, to the elliptic Dirichlet problem (\ref{Dirichlet
problem}); that is, to the elliptic Dirichlet problem 
\begin{equation}
\left\{ 
\begin{array}{l}
\sum_{i=1}^{n}\frac{\partial }{\partial x_{i}}a^{i}\left( x,u,Du\right)
=b\left( x,u,Du\right) ,\;\;\;\;\;x\in \Omega \,, \\ 
u=u_{0}\;\;\;\;\;\text{on}\;\;\partial \Omega \,.%
\end{array}%
\right.  \label{GCfe:Dirichlet problem}
\end{equation}%
For the sake of clarity we resume here the assumptions. The vector field $%
a\left( x,u,Du\right) :=\left( a^{i}\left( x,u,\xi \right) \right)
_{i=1,\ldots ,n}$ and the right-hand side $b\left( x,u,\xi \right) $ are
Carath\'{e}odory maps defined in $\Omega \times \mathbb{R}\times \mathbb{R}%
^{n}$; i.e. measurable with respect to $x\in \Omega $ and continuous in $%
\left( u,\xi \right) \in \mathbb{R\times R}^{n}$. The boundary datum $%
u_{0}:\Omega \rightarrow \mathbb{R}$ is bounded in $\Omega $. As in Section %
\ref{Section: Introduction and statements} we consider a Carath\'{e}odory
function $f=f\left( x,u,\xi \right) $, \textit{convex} with respect to the
gradient variable $\xi =\left( \xi _{i}\right) _{i=1,\ldots ,n}\in \mathbb{R}%
^{n}$ for a.e. $x\in \Omega $ and all $u\in \mathbb{R}$. This function $%
f\geq 0$ is \textit{coercive} with exponent $p\in (1,n]$ as in (\ref%
{coercivity condition}), it satisfies the \textit{two sided }$\Delta _{2}-$%
\textit{condition} (\ref{Delta-two (two sides)}) with respect to $\xi \in 
\mathbb{R}^{n}$, {the} condition (\ref{request on the u-dependence}) with
respect to{\ $u\in \mathbb{R}$ and the natural summability property (\ref%
{summability property})}. The Sobolev\textit{\ class} $W^{1,F}\left( \Omega
\right) $ where to look for weak solutions is defined in (\ref{the Sobolev
class}) and $W_{0}^{1,F}\left( \Omega \right) :=W^{1,F}\left( \Omega \right)
\cap W_{0}^{1,p}\left( \Omega \right) $. A weak solution to (\ref%
{GCfe:Dirichlet problem}) is a Sobolev function $u\in
u_{0}+W_{0}^{1,F}(\Omega )$ such that 
\begin{equation}
\int_{\Omega }\left\{ \sum_{i=1}^{n}a^{i}(x,u,Du)\varphi
_{x_{i}}+b(x,u,Du)\varphi \right\} \,dx=0\,,\;\;\;\;\forall \;\varphi \in
W_{0}^{1,F}(\Omega ).  \label{GCfe:aweaksol}
\end{equation}%
The vector field $a\left( x,u,Du\right) :=\left( a^{i}\left( x,u,\xi \right)
\right) _{i=1,\ldots ,n}$ satisfies the conditions 
\begin{equation}
\left\{ 
\begin{array}{l}
\left( a\left( x,u,\xi \right) ,\xi \right) \geq c_{1}\,\left( D_{\xi
}f\left( x,u,\xi \right) ,\xi \right) -c_{2}\left\vert u\right\vert
^{p^{\ast }}-b_{1}\left( x\right) \\ 
\left( a\left( x,u,\xi \right) ,\xi \right) \leq c_{3}\,\left( D_{\xi
}f\left( x,u,\xi \right) ,\xi \right) +c_{4}\left\vert u\right\vert
^{p^{\ast }}+b_{2}\left( x\right)%
\end{array}%
\right.  \label{GCfe:comparison inequalities}
\end{equation}%
for some positive constants $c_{i}$, for a.e. $x\in \Omega $, every $\left(
u,\xi \right) \in \mathbb{R}\times \mathbb{R}^{n}$, and for{\ nonnegative}
functions $b_{1},b_{2}$ with $b_{1}\in L^{s_{1}}(\Omega )$, $s_{1}>\tfrac{n}{%
p}$, $b_{2}\in L^{1}(\Omega )$.

The right-hand side $b\left( x,u,\xi \right) $ is a Carath\'{e}odory
function defined in $\Omega \times \mathbb{R}\times \mathbb{R}^{n}$ and
satisfying 
\begin{equation}
\left\{ 
\begin{array}{l}
\operatorname{sign}\left( u\right) \,b\left( x,u,\xi \right) \geq -\,c_{7}\,\big(%
f(x,u,\xi )^{\frac{p^{\ast }-1}{p^{\ast }}}+|\xi |^{p\frac{p^{\ast }-1}{%
p^{\ast }}}+\left\vert u\right\vert ^{p^{\ast }-1}\big)-b_{3}(x) \\ 
u\,b\left( x,u,\xi \right) \leq c_{8}\,\big(f\left( x,u,\xi \right)
+\left\vert u\right\vert ^{p^{\ast }}\big)+b_{4}(x)%
\end{array}%
\right.  \label{GCfe:growth condition for bpair}
\end{equation}%
for some positive constants $c_{i}\,$, for a.e. $x\in \Omega $, every $%
\left( u,\xi \right) \in \mathbb{R}\times \mathbb{R}^{n}$ and nonnegative
functions $b_{3},b_{4}$, such that $b_{3}\in L^{s_{3}}(\Omega )$, $s_{3}>%
\tfrac{n}{p}$, and $b_{4}\in L^{1}(\Omega )$.

If $p=n$, due to the arbitrariness in the choice of $p^{\ast }$, $\frac{%
p^{\ast }-1}{p^{\ast }}$ is any positive number in $(0,1)$ and the
assumptions on $s_{1},s_{3}$ reduce to $s_{1},s_{3}>1$.

\begin{remark}
\label{r:boundalternativoperb} A sufficient condition to (\ref{GCfe:growth
condition for bpair}) is 
\begin{equation*}
|b(x,u,\xi )|\leq c\,\big(f(x,u,\xi )^{\frac{p^*-1}{p^*} }+|\xi |^{p\frac{%
p^*-1}{p^*}}+\left\vert u\right\vert ^{p^*-1}+ b_3(x)\big)
\end{equation*}
for some $c>0$ and a nonnegative $b_3\in L^{s_3}(\Omega )$. Indeed, from
this inequality the first inequality in (\ref{GCfe:growth condition for
bpair}) trivially follows. As far as the second inequality in (\ref%
{GCfe:growth condition for bpair}) is concerned, from the H\"{o}lder
inequality we have 
\begin{align}
u\,b(x,u,\xi )\leq & \,c\,|u|\big(f(x,u,\xi )^{\frac{p^*-1}{p^*} }+|\xi |^{p%
\frac{p^*-1}{p^*}}+\left\vert u\right\vert ^{p^*-1}+ b_3(x)\big)  \notag \\
\leq & \,c\,\Big(f(x,u,\xi ) +|\xi |^{ p }+|u|^{p^{\ast }} +b_3^{\frac{%
p^{\ast }}{p^{\ast }-1}}\Big).  \label{r:stimalimub}
\end{align}
We observe that $|\xi |^{p}$ can be bounded from above by $f(x,u,\xi)$, due
to the coercivity condition (\ref{coercivity condition}). Moreover, due to
the bounds on the exponents listed above, the inequalities $s_3> \frac{%
p^{\ast }}{p^{\ast }-p}>\frac{p^{\ast }}{p^{\ast }-1}$ hold, hence the
right-hand side in (\ref{GCfe:growth condition for bpair}) is obtained with $%
b_4:=b_3^{\frac{p^{\ast }}{p^{\ast }-1}}\in L^{1}(\Omega )$.
\end{remark}

\bigskip

Before proving Theorem \ref{main boundeness result} we give some preliminary
remarks about the choice of the test function and the well posedness of (\ref%
{GCfe:aweaksol}).

\subsection{The test function and the well definition of weak solution\label%
{ss:testfunctwelldef}}

It is important to verify that the test-function $\varphi $ used in the
proof of Theorem \ref{main boundeness result} is in $W_{0}^{1,F}(\Omega )$
and that the integrals 
\begin{equation}
\int_{\Omega }(a(x,u,Du),D\varphi \left( x\right) )\,dx  \label{pairing1}
\end{equation}%
and 
\begin{equation}
\int_{\Omega }b(x,u,Du)\varphi \left( x\right) \,dx  \label{pairing2}
\end{equation}%
are well defined. Given the boundary datum $u_{0}\in L^{\infty }(\Omega
)\cap W^{1,F}(\Omega )$ and $u\in u_{0}+W_{0}^{1,F}(\Omega )$ a weak
solution to (\ref{GCfe:Dirichlet problem}), we will use, as test function $%
\varphi $ in (\ref{GCfe:aweaksol}), the function 
\begin{equation}
\varphi (x):=(\left\vert u(x)\right\vert -k)_{+}\operatorname{sgn}(u(x))\quad \text{
for a.e.}\ x\in \Omega \,,  \label{e:varphitest}
\end{equation}%
where $(|u(x)|-k)_{+}=\max \{|u(x)|-k,0\}$ and $k>\Vert u_{0}\Vert _{\infty
} $. We show that $\varphi \in W_{0}^{1,F}(\Omega )$.

\begin{remark}
\label{r:testinW1F} We claim that the function $\varphi $ in (\ref%
{e:varphitest}) belongs to $W_{0}^{1,F}(\Omega )$. First, we notice that $%
\varphi \in W_{0}^{1,p}(\Omega )$, since $u\in W^{1,p}(\Omega )$ and the
trace is $0$, because 
\begin{equation*}
\left\vert u(x)\right\vert -k<\left\vert u_{0}(x)\right\vert -\Vert
u_{0}\Vert _{\infty }\le 0\ \ \text{ on the boundary of $\Omega $}.
\end{equation*}
Let us prove that $\varphi \in W^{1,F}(\Omega )$. Define 
\begin{equation}
A_{k}:=\{x\in \Omega \,:\,\left\vert u(x)\right\vert >k\}.  \label{e:defAk}
\end{equation}
Then, a.e. in $\Omega $, 
\begin{equation}
\varphi =((\left\vert u\right\vert -k)\operatorname{sgn}(u))\,\chi _{A_{k}},\qquad
D\varphi =Du\,\chi _{A_{k}},  \label{e:propphi}
\end{equation}
therefore 
\begin{equation*}
F(\varphi )=\int_{A_{k}}f(x,(\left\vert u(x)\right\vert -k)\operatorname{sgn}
(u(x)),Du)\,dx+\int_{\Omega \setminus A_{k}}f(x,0,0)\,dx.
\end{equation*}
The last integral is finite because of (\ref{summability property}). The
first integral can be estimated by using (\ref{request on the u-dependence}%
). Indeed, for a.e. $x\in A_{k}$ the following inequalities hold: 
\begin{equation*}
\text{if $u(x)>0$:}\qquad 0<(\left\vert u(x)\right\vert -k)\operatorname{sgn}
(u(x))=u(x)-k\leq u(x);
\end{equation*}
\begin{equation*}
\text{if $u(x)<0$:}\qquad 0>(\left\vert u(x)\right\vert -k)\operatorname{sgn}
(u(x))=k+u(x)\geq u(x).
\end{equation*}
Therefore, by (\ref{request on the u-dependence}), 
\begin{equation*}
\int_{A_{k}}f(x,(\left\vert u\right\vert -k)\operatorname{sgn}(u),Du)\,dx\leq
c_6\int_{A_{k}}f(x,u,Du)\,dx
\end{equation*}
and this last integral is finite, because $u\in W^{1,F}(\Omega )$.
\end{remark}

We now discuss the well posedness of the definition of weak solution in
relationship with our summability assumptions.

\begin{remark}[about the vector field $a$]
\label{r:a<p*} Precisely, we discuss the well posedness of 
\begin{equation*}
\int_{\Omega }(a(x,u,Du),D\varphi \left( x\right) )\,dx\,,
\end{equation*}
where $\varphi$ is the test function in (\ref{e:varphitest}). 
%	To do this we use the inequalities (\ref{GCfe:comparison inequalities}) and the right inequality in  (\ref{gradient growth (two-sides)}). 
From (\ref{GCfe:comparison inequalities}) we obtain 
\begin{align}
\int_{\Omega} |(a(x,u,Du),Du \left( x\right))|\,dx \le &\, c \int_{\Omega}
\left( D_{\xi }f\left( x,u,Du \right) ,Du \right)\,dx  \notag \\
&+ c\int_{\Omega}\left( \left\vert u\right\vert ^{p^*}+ b_{1}\left(
x\right)+ b_{2}\left( x\right)\right)\,dx,  \label{e:postpairinga}
\end{align}
for a positive constant depending on $c_i$, $i=1,\ldots,4$. The first
integral at the right-hand side is finite, due to the right inequality in (%
\ref{gradient growth (two-sides)}) in Lemma \ref{Lemma 2}, indeed 
\begin{equation*}
\int_{\Omega} \left( D_{\xi }f\left( x,u,Du \right) ,Du \right)\,dx\le
(M-1)\int_{\Omega}f(x,u,Du)\,dx\,,
\end{equation*}
which is finite because $u\in W^{1,F}(\Omega)$. As far as the last integral
in (\ref{e:postpairinga}) is concerned, we first notice that, by
assumptions, $b_1,b_2\in L^1(\Omega)$. To conclude, it suffices to remark
that, since $u\in W^{1,F}(\Omega)$, then, by (\ref{coercivity condition}),
it is $u\in W^{1,p}(\Omega)$ and, in particular, $u\in L^{p^*}(\Omega)$.
\end{remark}

%
%	\begin{equation}
%		\left\{ 
%		\begin{array}{l}
%			\left( a\left( x,u,\xi \right) ,\xi \right) \geq c_{1}\,\left( D_{\xi
%			}f\left( x,u,\xi \right) ,\xi \right) -c_{2}\left\vert u\right\vert ^{\theta
%				_{1}}-b_{1}\left( x\right) \\ 
%			\left( a\left( x,u,\xi \right) ,\xi \right) \leq c_{3}\,\left( D_{\xi
%			}f\left( x,u,\xi \right) ,\xi \right) +c_{4}\left\vert u\right\vert ^{\theta
%				_{2}}+b_{2}\left( x\right)%
%		\end{array}%
%		\right.  \label{comparison inequalities}
%	\end{equation}%
%	

\begin{remark}[about the datum $b$]
\label{r:datumb} We claim that if $\varphi $ is the test function in (\ref%
{e:varphitest}), then $x\mapsto b(x,u,Du) \varphi \left( x\right)$ is in $%
L^1(\Omega)$. 
%By Remark \ref{r:testinW1F}, to be in $W_{0}^{1,F}\left( \Omega \right)$. 
We need to prove that for every $k>0$ 
\begin{equation}
b(x,u,Du)u\in L^1(A_k), \qquad b(x,u,Du)\operatorname{sgn}(u)\in L^1(A_k),
\label{GCfe:bb}
\end{equation}
where $A_k:=\{|u|>k\}$. To prove the first summability property it is
sufficient, due to (\ref{GCfe:growth condition for bpair}), to prove that
both 
\begin{equation}
x\mapsto f(x,u,Du )+\left\vert u\right\vert ^{p^{\ast }}+b_4(x)
\label{GCe:righthandsidebuL1}
\end{equation}
and 
\begin{equation}
x\mapsto |u|\big(f(x,u,Du )^{\frac{p^*-1}{p^*} }+|Du |^{p\frac{p^*-1}{p^*}%
}+\left\vert u\right\vert ^{p^*-1}+ b_3(x)\big)  \label{GCe:lefthandsidebuL1}
\end{equation}
are in $L^1(\Omega)$.

The function in (\ref{GCe:righthandsidebuL1}) is trivially in $L^1(\Omega)$,
because $u\in W^{1,F}(\Omega)\subseteq W^{1,p}(\Omega)$. Let us consider the
function in (\ref{GCe:lefthandsidebuL1}). As in Remark \ref%
{r:boundalternativoperb}, 
\begin{align*}
|u|(f(x,u,Du )^{\frac{p^*-1}{p^*} }+&|Du |^{p\frac{p^*-1}{p^*}}+\left\vert
u\right\vert ^{p^*-1}+ b_3(x)) \\
\le & \,c\,\big( f(x,u,Du) +|Du|^{p}+|u|^{p^*} +b_3^{\frac{p^{\ast }}{%
p^{\ast }-1}}\big)
\end{align*}
and, due to the bound on the summability exponent $s_3$ of $b_3$, we
conclude that (\ref{GCe:lefthandsidebuL1}) is in $L^1(\Omega)$.

As far as the second summability condition in (\ref{GCfe:bb}) is concerned,
we use the estimates (\ref{GCfe:growth condition for bpair}) once again to
prove that the function is bounded from below and from above by functions in 
$L^1(A_k)$. The function 
\begin{equation*}
x\mapsto f(x,u,Du )^{\frac{p^*-1}{p^*} }+|Du |^{p\frac{p^*-1}{p^*}%
}+\left\vert u\right\vert ^{p^*-1}+ b_3(x)
\end{equation*}
can be estimated, up to a constant, by 
\begin{equation*}
x\mapsto f(x,u,Du)+\left\vert Du\right\vert ^{p}+\left\vert u\right\vert
^{p^*}+b_3(x)+1
\end{equation*}
that is in $L^1(\Omega)$. On the other hand, by the second inequality in (%
\ref{GCfe:growth condition for bpair}), for every $x\in A_k$ 
\begin{align*}
b(x,u,Du)\operatorname{sgn}(u)\le & \,\frac{1}{|u|} \big(c_8 f(x,u,Du)+c_8\left\vert
u\right\vert ^{p^{\ast }}+b_4(x)\big) \\
\le & \,\frac{1}{k} \big(c_8 f(x,u,Du)+c_8\left\vert u\right\vert ^{p^{\ast
}}+b_4(x)\big)
\end{align*}
and the right-hand side is in $L^1(A_k)$.
\end{remark}

\subsection{The Caccioppoli-type inequality}

In this section we prove a Caccioppoli-type inequality for weak solutions of
the Dirichlet problem (\ref{GCfe:Dirichlet problem}). In the following, as
before, $A_{k}:=\{x\in \Omega \,:\,|u(x)|>k\}$.

\begin{proposition}
Let $u_{0}\in L^{\infty }(\Omega )\cap W^{1,F}(\Omega )$ and let $u\in
u_{0}+W_{0}^{1,F}(\Omega )$ be a weak solution to (\ref{GCfe:Dirichlet
problem}) under the assumptions of Theorem \ref{main boundeness result}.
Then there exists $c>0$, depending on the data, but not on $u$, such that
for every $k\in \mathbb{R}$, $k>\max\{\Vert u_{0}\Vert _{\infty },1\}$, 
\begin{align}
\int_{A_{k}}|Du|^{p}\,dx\leq & \,c\Vert b_3\Vert _{L^{s_3}(\Omega )}\Vert
|u|-k\Vert _{L^{p^{\ast }}(A_{k})}|A_{k}|^{1-\frac{1}{s_3}-\frac{1 }{
p^{\ast }}}  \notag \\
& +c\,\Vert |u|-k\Vert _{L^{p^{\ast }}(A_{k})}^{p^{\ast } }  \notag \\
& +c\,k^{p^{\ast }}|A_{k}|+c\Vert b_{1}\Vert _{L^{s_{1}}(\Omega
)}|A_{k}|^{1- \frac{1}{s_{1}}}.  \label{e:fineIIIstep}
\end{align}
\label{GCfp:Caccioppoli}
\end{proposition}

\begin{proof}
For every $k>\max\{\Vert u_{0}\Vert _{\infty },1\}$, we define the test
function $\varphi _{k}$ as follows 
\begin{equation*}
\varphi _{k}(x):=(|u(x)|-k)_{+}\operatorname{sgn}(u(x))\quad \text{for a.e.}\ x\in
\Omega.
\end{equation*}
Notice that $\varphi _{k}\in W_{0}^{1,F}(\Omega )$, see Remark \ref%
{r:testinW1F}.

Let us consider the super-level sets: $A_{k}:=\{x\in \Omega \,:\,|u(x)|>k\}$
. Using $\varphi _{k}$ as a test function in (\ref{GCfe:aweaksol}) we get 
\begin{equation}
\begin{aligned} I_1:=& \int_{A_{k}} (a(x,u,Du),Du)\,dx \\=& \int_{A_{k}} -
b(x,u,Du) (|u|-k)\operatorname{sgn}(u)\,dx =:I_2. \end{aligned}
\label{NEWaveroiniziostep1}
\end{equation}
Now, we separately consider and estimate $I_{i}$, $i=1,2$.

\medbreak
\textbf{Estimate of $I_1$.}

\medbreak
By the first inequalities in (\ref{GCfe:comparison inequalities}) and (\ref%
{gradient growth (two-sides)}), 
\begin{align}
I_1\ge &\,\int_{A_k}\left\{c_{1}\left( D_{\xi }f\left( x,u,Du\right) ,Du
\right) -c_{2}\left\vert u\right\vert ^{p^*}-b_{1}\left( x\right)\right\}\,dx
\notag \\
\ge&\, \int_{A_k}\left\{2c_1 \left( 1-\tfrac{1}{m}\right) \,f\left(
x,u,Du\right)-c(\left\vert u\right\vert-k) ^{p^*}-ck^{p^*}-b_{1}\left(
x\right)\right\}\,dx\,,  \label{GCfe:I1stimasotto}
\end{align}
for a positive constant $c$. 
%2c_1 \left( 1-\tfrac{1}{m}\right) \,f\left( x,u,\xi \right)
% and by the coercivity
%condition (\ref{coercivity condition}) we get 
%\begin{align}
%I_1\ge &\,\int_{A_k}\left\{c_{1}\left( D_{\xi }f\left( x,u,Du\right) ,Du
%\right) -c_{2}\left\vert u\right\vert ^{\theta}-b_{1}\left(
%x\right)\right\}\,dx  \notag \\
%\ge & \int_{A_k} \left\{ c_1c_{5} 2\left( 1-\tfrac{1}{m}\right)\,\left\vert
%Du\right\vert ^{p} -c_{2}\left\vert u\right\vert ^{\theta}-b_{1}\left(
%x\right)\right\}\,dx.  \label{GCfe:I1stimasotto}
%\end{align}

\medbreak
\textbf{Estimate of $I_2$.}

\medbreak
By the first inequality in (\ref{GCfe:growth condition for bpair}) we have 
\begin{equation*}
-b(x,u,\xi)\operatorname{sgn}(u)\le c_7\,\big(f(x,u,\xi)^{\frac{p^*-1}{p^*}%
}+\left\vert \xi\right\vert ^{p\frac{p^*-1}{p^*}}+\left\vert u\right\vert
^{p^*-1}\big)+b_3(x),
\end{equation*}
therefore 
\begin{equation*}
I_{2}\leq \int_{A_{k}}\big\{c_7\,\big( f(x,u,Du)^{\frac{p^*-1}{p^*}}+|Du|^{p%
\frac{p^*-1}{p^*}}+|u|^{p^*-1}\big)+b_3(x)\big\}(|u|-k) \,dx.
\end{equation*}
We estimate the right-hand side using the Young inequality; thus there
exists $c>0$, such that 
\begin{align*}
& c_7\,\big(f(x,u,Du)^{\frac{p^*-1}{p^*}}+|Du|^{p\frac{p^*-1}{p^*}}\big)%
(|u|-k) \\
\le & \,\frac{c_1c_5}{c_5+1}\left( 1-\tfrac{1}{m}\right) \big\{f\left(
x,u,Du\right)+|Du|^p\}+c\,(|u|-k)^{p^*} \\
\le &\,c_1\left( 1-\tfrac{1}{m}\right) f\left(
x,u,Du\right)+c\,(|u|-k)^{p^*},
\end{align*}
where $c_5$ is the coefficient in the coercivity condition (\ref{coercivity
condition}), a property used in the last inequality. Thus, we get that there
exists $c>0$, such that 
\begin{align*}
I_{2}\leq & \,c_1\left( 1-\tfrac{1}{m}\right)\int_{A_{k}}f\left(
x,u,Du\right)\,dx+c \int_{A_{k}}(|u|-k)^{p^*} \,dx  \notag \\
& +\int_{A_{k}}\big\{c_7|u|^{p^*-1}+b_3(x)\big\}(|u|-k)\,dx\,.
\end{align*}%
For a.e. $x\in A_{k}$, we have $|u|^{p^*-1}\leq c\,(|u|-k)^{p^*-1}
+c\,k^{p^*-1}$; therefore, using the Young inequality once again, 
\begin{align}
I_{2}\leq & \,c_1\left( 1-\tfrac{1}{m}\right)\int_{A_{k}}f\left(
x,u,Du\right)\,dx+c \int_{A_{k}}\big((|u|-k)^{p^*}+k^{p^*}\big) \,dx  \notag
\\
& +\int_{A_{k}}b_3(x)(|u|-k)\,dx\,.  \label{I_4I}
\end{align}
Collecting (\ref{NEWaveroiniziostep1}), (\ref{GCfe:I1stimasotto}), (\ref%
{I_4I}) and using the coercivity condition (\ref{coercivity condition}), we
get 
\begin{align}
\int_{A_{k}}|Du|^{p}\,dx \le c\,
\int_{A_{k}}\left\{(|u|-k)^{p^*}+b_3(x)(|u|-k)+b_{1}(x)+k^{p^*}\right\} \,dx.
\label{e:radice}
\end{align}
for some $c>0$ depending on $n,p,m,c_1,c_5,c_7$. Since the integrability
exponents $s_1$,$s_3$ of $b_1$,$b_{3}$ satisfy $s_1>1 $ and $s_3> \frac{p^*}{%
p^*-p}>\frac{p^*}{p^*-1}$, then 
\begin{equation*}
\int_{A_{k}}b_3(x)(|u|-k)\,dx\leq \,c\,\Vert b_3\Vert _{L^{s_3}(\Omega
)}\Vert |u|-k\Vert _{L^{p^{\ast }}(A_{k})}|A_{k}|^{1-\frac{1}{s_3}-\frac{1 }{
p^{\ast }}}
\end{equation*}
and 
\begin{equation*}
\int_{A_{k}}b_{1}(x)\,dx\leq \Vert b_{1}\Vert _{L^{s_{1}}(\Omega
)}|A_{k}|^{1- \frac{1}{s_{1}}}.
\end{equation*}

Collecting (\ref{e:radice}) and these estimates we obtain the Caccioppoli's
inequality (\ref{e:fineIIIstep}).
\end{proof}

\subsection{The iteration scheme}

Given any real number $d>2\max\{\Vert u_{0}\Vert _{\infty },1\}$ 
%$d>2\Vert u_{0}\Vert _{L^{\infty }(\Omega )}+1$,
we consider the increasing sequence 
\begin{equation}
k_{h}:=d(1-\tfrac{1}{2^{h+1}}),\qquad h\in \mathbb{N}\cup \{0\}
\label{eq.defbetah}
\end{equation}%
and the sequence $(J_{h})_{h\geq 0}$ of nonnegative numbers as follows: 
\begin{equation}
J_{h}:=\int_{A_{k_{h}}}|Du|^{p}\,dx.  \label{eq.defJh}
\end{equation}
Since, for every $h$, $(|u|-k_{h})_{+}\operatorname{sgn}(u)$ is in $%
W_{0}^{1,F}(\Omega )\subseteq W_{0}^{1,p}(\Omega )$ (see Remark \ref%
{r:testinW1F}), then $J_h$ is finite and, by Poincar\'e inequality, there
exists $C_P$, independent of $h$ and $u$, such that 
\begin{equation}  \label{e:poincareJh}
\int_{A_{k_{h}}}(|u|-k_h)^{p^*}\,dx=\|(|u|-k_h)_+\|_{L^{p^*}(\Omega)}^{p^*}
\le c_P\,J_h^{\frac{p^*}{p}},
\end{equation}%
where we used that $|D(|u|)|=|Du|$ almost everywhere in $A_{k_h}$. 
\medbreak

\bigskip

\begin{proposition}
\label{prop.Jhestimate} For every $h\in \mathbb{N}\setminus\{0\}$, 
\begin{equation}
J_{h+1}\leq \,c_{*}\, (1+\Vert Du\Vert _{L^{p}(\Omega )}^{p^*})^{\max \{%
\frac{ 1}{ s_1},\frac{1}{s_{3}}\}} \, 2^{p^{\ast }h} J_{h}^{\frac{p^*}{p}%
(1-\max \{\frac{1}{s_1},\frac{1}{ s_{3}}\})},  \label{e:Jhp1leJh}
\end{equation}
where $c_{\ast }$ is a positive constant depending on the data, the $L^{s_1}$
and the $L^{s_{3}}$ norms of $b_1,b_{3}$, respectively, but it is
independent of $u$ and $d$.
\end{proposition}

We notice that, by the assumptions on the parameters, we have 
\begin{equation*}
\tfrac{p^{\ast }}{p}\Big(1-\max \big\{\tfrac{1}{s_{1}},\tfrac{1}{s_{3}}\big\}%
\Big)>1\,.
\end{equation*}

\begin{proof}
Since $(k_{h})_{h}$ is increasing, then the sequence $(J_{h})_{h}$ is
decreasing. Moreover, by taking into account the definitions of $J_{h}$ and $%
k_{h}$ we have, by (\ref{e:poincareJh}), 
\begin{equation}
\begin{split}
c_P\,J_{h}^{\frac{p^*}{p}}& \ge \int_{A_{k_{h}}}(|u|-k_{h})^{p^{\ast
}}\,dx\geq \int_{A_{k_{h+1}}}(|u|-k_{h})^{p^{\ast }}\,dx \\
& \geq (k_{h+1}-k_{h})^{p^{\ast }}\,\big|A_{k_{h+1}}\big|=\bigg(\frac{d}{
2^{h+2}}\bigg)^{p^{\ast }}\,\big|A_{k_{h+1}}\big|\,,
\end{split}
\label{eq.estimJhgeq}
\end{equation}
so we get 
\begin{equation}
\big|A_{k_{h+1}}\big|\leq c_P \,4^{p^{\ast }}\,\big(\tfrac{2^{h}}{d}\big)%
^{p^{\ast }}J_{h}^{\frac{p^*}{p}}.  \label{abetah+1}
\end{equation}
%	Taking into account that $|D(|u|)|=|Du|$ a.e. in $A_{k_{h+1}}$, 
%	\begin{align}
%		J_{h+1}^{\frac{p}{p^{\ast }}}& =\left(
%		\int_{A_{k_{h+1}}}(|u|-k_{h+1})^{p^{\ast }}\,dx\right) ^{\frac{p}{p^{\ast }}%
%		} = \left( \int_{\Omega }((|u|-k_{h+1})_{+})^{p^{\ast }}\,dx\right) ^{\frac{%
%				p }{ p^{\ast }}}  \notag \\
%		& \leq C_{S}^{p}\int_{\Omega }\big|D((|u|-k_{h+1})_{+})\big| %
%		^{p}\,dx=C_{S}^{p}\int_{A_{k_{h+1}}}|Du|^{p}\,dx,  \label{e:etahbarrho2}
%	\end{align}
%	where $c_{S}$ is the Sobolev constant. To estimate the last integral in (\ref	{e:etahbarrho2}),
Therefore, by the Caccioppoli estimate (\ref{e:fineIIIstep}) and by (\ref%
{e:poincareJh}), 
\begin{align*}
J_{h+1}=&\int_{A_{k_{h+1}}}|Du|^{p}\,dx\leq \,c\,\Vert b_3\Vert
_{L^{s_3}(\Omega )}J_{h+1}^{\frac{1}{p}}|A_{k_{h+1}}|^{1-\frac{1}{ s_3}-%
\frac{1}{p^{\ast }}}  \notag \\
& +c\,J_{h+1}^{\frac{p^*}{p}} +c\,k_{h+1}^{p^*}|A_{k_{h+1}}|+c\,\Vert
b_{1}\Vert _{L^{s_{1}}(\Omega )}|A_{k_{h+1}}|^{1-\frac{1}{s_{1}}},
\end{align*}
that implies, by (\ref{abetah+1}), by using that $(J_h)$ is a decreasing
sequence and $k_{h+1}\leq d$, 
\begin{align}
J_{h+1}\leq &\,c\,\Vert b_3\Vert _{L^{s_3}(\Omega )}\big(\tfrac{2^{h}}{d}%
\big)^{p^{\ast }-\frac{p^{\ast } }{ s_3}-1}J_{h}^{\frac{p^*}{p}\big(1-\frac{1%
}{s_3}\big)}  \notag \\
& +c\,2^{p^*h} J_{h}^{\frac{p^*}{p}} +c\,\Vert b_{1}\Vert _{L^{s_{1}}(\Omega
)}\big(\tfrac{2^{h}}{d}\big) ^{p^{\ast }-\frac{p^{\ast }}{s_{1}}}J_{h}^{%
\frac{p^*}{p}\big(1-\frac{1}{s_{1}}\big)},  \label{e:postusoIIIstep2}
\end{align}
with a constant $c$ independent of $d$, $h$ and $u$. Notice that $J_{h}\leq
\Vert D u\Vert _{L^{p}(\Omega )}^{p}$ for every $h\in \mathbb{N}$, so that 
\begin{equation*}
\max \{J_{h};\;J_{h}^{1-\frac{1}{s_1}};\;J_{h}^{1-\frac{1}{s_{3}}}\}\leq %
\Big(1+\Vert D u\Vert _{L^{p}(\Omega )}^{p}\Big)^{\max \{\frac{ 1}{ s_1},%
\frac{1}{s_{3}}\}}J_{h}^{1-\max \{\frac{1}{s_1},\frac{1}{ s_{3}} \}}.
\end{equation*}
If we denote 
\begin{equation}  \label{d:sigma}
\sigma :=p^{\ast }-\max \left\{\tfrac{p^{\ast }}{s_{1}};\;\tfrac{ p^{\ast }}{%
s_3}+1 \right\} ,
\end{equation}
then $\sigma>0$, due to the assumptions on $s_1,s_3$. Therefore, by
inequality (\ref{e:postusoIIIstep2}) and taking into account that $d>1$, we
obtain 
\begin{align*}
J_{h+1}\leq & \,c_{0}\left( 1+\Vert b_{1}\Vert _{L^{s_{1}}(\Omega )}+\Vert
b_3\Vert _{L^{s_3}(\Omega )}\right) \times \\
& \qquad \qquad \times \big(2^{p^{\ast }}\big) ^{h}(1+\Vert Du\Vert
_{L^{p}(\Omega )}^{p^*})^{\max \{\frac{ 1}{ s_1},\frac{1}{s_{3}}\}} J_{h}^{%
\frac{p^*}{p}\big(1-\max \{\frac{1}{s_1},\frac{1}{ s_{3}}\}\big)}.
\end{align*}
Therefore we get (\ref{e:Jhp1leJh}) with 
\begin{equation}
c_{\ast }:=c_0\, \left( 1+\Vert b_{1}\Vert _{L^{s_{1}}(\Omega )}+\Vert
b_3\Vert _{L^{s_3}(\Omega )}\right).  \label{e:Jcstar}
\end{equation}
\end{proof}

\bigskip

\subsection{Conclusion of the proof of Theorem \protect\ref{main boundeness
result}}

We are ready to conclude the proof of the global boundedness result stated
in Theorem \ref{main boundeness result}. Let $u\in u_{0}+W_{0}^{1,F}(\Omega
) $ be a weak solution to (\ref{GCfe:Dirichlet problem}), with $u_{0}\in
L^{\infty }(\Omega )\cap W^{1,F}(\Omega )$. Let $d>2\max \{\Vert u_{0}\Vert
_{\infty },1\}$ and let $(J_{h})_{h\geq 0}$ be the sequence defined in (\ref%
{eq.defJh}). Owing to Proposition \ref{prop.Jhestimate}, we have the
estimate 
\begin{equation}
J_{h+1}\leq L\,(2^{p^{\ast }})^{h}\,J_{h}^{1+\delta }\qquad (h\in \mathbb{N}%
\cup \{0\}),  \label{e:4.28'}
\end{equation}%
where $\delta $ is 
\begin{equation}
\delta :=\tfrac{p^{\ast }}{p}\big(1-\max \{\tfrac{1}{s_{1}},\tfrac{1}{s_{3}}%
\}\big)-1,  \label{eq.alpha}
\end{equation}%
which is positive because $s_{1},s_{3}>\frac{n}{p}$, and 
\begin{equation*}
L:=c_{\ast }\,(1+\Vert Du\Vert _{L^{p}(\Omega )}^{p^{\ast }})^{\max \{\frac{1%
}{s_{1}},\frac{1}{s_{3}}\}},
\end{equation*}%
where $c_{\ast }$, independent of $d$, is defined as in (\ref{e:Jcstar}).
Since $u\in W^{1,p}(\Omega )$, it is possible to choose $d$ in such a way
that 
\begin{equation}
J_{0}:=\int_{A_{\frac{d}{2}}}|Du|^{p}\,dx\leq L^{-\frac{1}{\delta }%
}\,(2^{p^{\ast }})^{-\frac{1}{\delta ^{2}}}.  \label{eq.claimJzero}
\end{equation}%
With (\ref{eq.claimJzero}) at hand, we are entitled to apply a well known
classical lemma of Real Analysis (see, e.g., \cite[ Lemma 7.1]{Giusti 2003
book}), which says that if $(z_{h})_{h\geq 0}$ is a sequence of positive
real numbers satisfying the fol\-low\-ing recursive relation 
\begin{equation}
z_{h+1}\leq L\,\zeta ^{h}z_{h}^{1+\delta }\qquad (h\in \mathbb{N}\cup \{0\}),
\label{eq.estimzhlemma}
\end{equation}%
where $L,\delta >0$ and $\zeta >1$, and if $z_{0}\leq L^{-\frac{1}{\delta }%
}\zeta ^{-\frac{1}{\delta ^{2}}}$, then $z_{h}\leq \zeta ^{-\frac{h}{\delta }%
}z_{0}\quad $for every $h\geq 0$. In particular, $z_{h}\rightarrow 0$ as $%
h\rightarrow \infty $. From this result, (\ref{e:4.28'}) and (\ref%
{eq.claimJzero}), we obtain 
\begin{equation*}
\int_{A_{d}}|Du|^{p}\,dx=\lim_{h\rightarrow \infty
}\int_{A_{k_{h}}}|Du|^{p}\,dx=\lim_{h\rightarrow \infty }J_{h}=0\,.
\end{equation*}
Since $u\in u_0+W_0^{1,p}(\Omega)$ and $d>\Vert u_{0}\Vert _{L^{\infty
}(\Omega )}$, then $(|u|-d)_+\in W_{0}^{1,p}(\Omega)$; therefore, 
\begin{equation*}
\int_\Omega(|u|-d)_+^{p^*}\,dx\le c_P\Big(\int_{A_{d}}|Du|^{p}\,dx\Big)^{%
\frac{p^*}{p}}=0,
\end{equation*}
that implies $|u|\le d$ a.e. We have so proved that $u\in L^{\infty }(\Omega
)$.

\subsection{Sketch of the proof of Theorem \protect\ref{main boundeness
result with epsilon}}

\label{s:dimconepsilon} In this section, we provide a sketch of the proof of
Theorem \ref{main boundeness result with epsilon}; in particular we give
details on how to derive the proof of Theorem \ref{main boundeness result
with epsilon} from that of Theorem \ref{main boundeness result}.

Unlike Theorem \ref{main boundeness result}, about the vector field $a\left(
x,u,Du\right) :=\left( a^{i}\left( x,u,\xi \right) \right) _{i=1,\ldots ,n}$%
, according to (\ref{comparison inequalities with epsilon})$_{1}$ we now
assume 
\begin{equation}
\left( a\left( x,u,\xi \right) ,\xi \right) \geq c_{1}\,\left( D_{\xi
}f\left( x,u,\xi \right) ,\xi \right) -c_{2}\left\vert u\right\vert ^{\theta
}-b_{1}\left( x\right) ,  \label{GCfe:comparison inequalitiesepsilon}
\end{equation}%
for an exponent $\theta \in \lbrack 0,p^{\ast })$; in Theorem \ref{main
boundeness result} we assumed $\theta =p^{\ast }$, see the first inequality
in (\ref{comparison inequalities}). As far as the right-hand side $b\left(
x,u,\xi \right) $ is concerned, in place of the first inequality in (\ref%
{growth condition for b}), according to (\ref{growth condition for b with
epsilon})$_{1}$ we now assume 
\begin{equation}
\operatorname{sign}\left( u\right) \,b\left( x,u,\xi \right) \geq -\,c_{7}\,\big(%
f(x,u,\xi )^{\alpha }+|\xi |^{r-1}+\left\vert u\right\vert ^{s-1}\big)%
+b_{3}(x)  \label{GCfe:growth condition for bpairconepsilon}
\end{equation}%
for some exponents ${\alpha \in \big[0,1-\tfrac{1}{p^{\ast }}}\big)$, $r\in {%
\big[}1,p+\tfrac{p}{n}\big)$, $s\in \lbrack 1,p^{\ast })$; therefore, now we
assume a more restrictive condition, since in (\ref{growth condition for b})
we had $\alpha =1-\tfrac{1}{p^{\ast }}$, $r=p+\tfrac{p}{n}$, $s=p^{\ast }$.

The scheme of the proof of Theorem \ref{main boundeness result with epsilon}
is formally the same of that of Theorem \ref{main boundeness result};
however, under these more restrictive assumptions we gain the possibility to
reach an explicit estimate of the $L^{\infty }$-norm of the solution $u$,
see (\ref{estimate in the main theorem with epsilon}).

Due to these different assumptions, in the proof of the Caccioppoli
inequality we obtain an estimate that is slightly more complicated than (\ref%
{e:radice}): 
\begin{align}
\int_{A_{k}}|Du|^{p}\,dx\leq & \,c\,\int_{A_{k}}\big\{(|u|-k)^{\theta
}+(|u|-k)^{s}+(|u|-k)^{\frac{1}{1-\alpha }}+(|u|-k)^{\frac{p}{p-r+1}}\big\}dx
\notag \\
& +c\,\int_{A_{k}}\left\{ b_{3}(x)(|u|-k)+b_{1}(x)+k^{\theta }+k^{s}\right\}
\,dx.
\end{align}%
Since $\theta ,s<p^{\ast }$ and $p^{\ast }(\alpha -1)>1$ and $r<p+\tfrac{p}{n%
}$, we can use the H\"{o}lder inequalities with exponents $\frac{p^{\ast }}{%
\theta }$, $\frac{p^{\ast }}{s}$, $p^{\ast }(1-\alpha )$ and $p^{\ast }\frac{%
p-r+1}{p}$, and the integrability assumptions on $b_{1}$,$b_{3}$, to get the
Caccioppoli-type inequality 
\begin{align*}
\int_{A_{k}}|Du|^{p}\,dx\leq & \,c\Vert b_{3}\Vert _{L^{s_{3}}(\Omega
)}\Vert |u|-k\Vert _{L^{p^{\ast }}(A_{k})}|A_{k}|^{1-\frac{1}{s_{3}}-\frac{1%
}{p^{\ast }}} \\
& +c\,\Vert |u|-k\Vert _{L^{p^{\ast }}(A_{k})}^{\theta }|A_{k}|^{1-\frac{%
\theta }{p^{\ast }}}+c\,\Vert |u|-k\Vert _{L^{p^{\ast
}}(A_{k})}^{s}|A_{k}|^{1-\frac{s}{p^{\ast }}} \\
& +c\Vert |u|-k\Vert _{L^{p^{\ast }}(A_{k})}^{\frac{1}{1-\alpha }}|A_{k}|^{1-%
\frac{1}{p^{\ast }(1-\alpha )}} \\
& +c\Vert |u|-k\Vert _{L^{p^{\ast }}(A_{k})}^{\frac{p}{p-r+1}}|A_{k}|^{1-%
\frac{p}{p^{\ast }(p-r+1)}} \\
& +c\,(k^{\theta }+k^{s})|A_{k}|+c\Vert b_{1}\Vert _{L^{s_{1}}(\Omega
)}|A_{k}|^{1-\frac{1}{s_{1}}}.
\end{align*}%
This inequality replaces (\ref{e:fineIIIstep}) and it is the starting point
to activate an iteration procedure, involving a sequence $(J_{h})_{h\geq 0}$
of nonnegative numbers. This sequence is now defined not as in (\ref%
{eq.defJh}), but as follows: 
\begin{equation*}
J_{h}:=\int_{A_{k_{h}}}(|u|-k_{h})^{p^{\ast }}\,dx.
\end{equation*}%
What it can be proved is that 
\begin{align*}
J_{h+1}\leq & \ c_{\ast }\left( 1+\Vert u\Vert _{L^{p^{\ast }}(\Omega
)}^{p^{\ast }}\right) ^{\frac{p^{\ast }}{p}\max \left\{ \frac{1}{s_{1}},%
\frac{1}{s_{3}}\right\} }\times \\
& \times \tfrac{1}{d^{\frac{p^{\ast }}{p}\sigma }}\left( 2^{\frac{p^{\ast }}{%
p}p^{\ast }}\right) ^{h}J_{h}^{\frac{p^{\ast }}{p}\left( 1-\max \left\{ 
\frac{1}{s_{1}},\frac{1}{s_{3}}\right\} \right) },
\end{align*}%
where $\sigma :=p^{\ast }-\max \big\{\tfrac{1}{1-\alpha };\;\tfrac{p}{p-r+1}%
;\,\theta ;\;s;\;\tfrac{p^{\ast }}{s_{1}};\;\tfrac{p^{\ast }}{s_{3}}+1\big\}$
, and $c_{\ast }$ is a positive constant depending on the data, the $%
L^{s_{1}}$ and the $L^{s_{3}}$ norms of $\,b_{1},b_{3}$, respectively, but
it is independent of $u$ and $d$. This inequality, similar to (\ref%
{eq.estimzhlemma}), allows to conclude, thanks to a well know classical
lemma of Real Analysis, as stated for instance in \cite[Lemma 7.1]{Giusti
2003 book}.

\bigskip

\bigskip

\textbf{Acknowledgement} \ The authors are members of the \textit{Gruppo
Nazionale per l'Analisi Matematica, la Probabilit\`{a} e le loro
Applicazioni (GNAMPA)} of the \textit{Istituto Nazionale di Alta Matematica
(INdAM)}. G. Cupini acknow\-ledges financial support under the National
Recovery and Resilience Plan (NRRP), Mission 4, Component 2, Investment 1.1,
Call for tender No. 104 published on 2.2.2022 by the Italian Ministry of
University and Research (MUR), funded by the European Union -
NextGenerationEU - Project PRIN\_CITTI 2022 - Title \textquotedblleft
Regularity problems in sub-Riemannian structures\textquotedblright\ - CUP
J53D23003760006 - Bando 2022 - Prot. 2022F4F2LH, and also through the
INdAM-GNAMPA Project CUP E53C23001670001 - \emph{Interazione ottimale tra la
regolarit\`{a} dei coefficienti e l'anisotropia del problema in funzionali
integrali a crescite non standard}.

\bigskip

\bigskip \textbf{Data availability} \ \ The manuscript has no associated
data.

\textbf{Declarations}

\textbf{Conflict of interest} \ \ On behalf of all authors, the
corresponding author states that there is no conflict of interest.


\bigskip

\begin{thebibliography}{99}
\bibitem{Baroni-Colombo-Mingione 2018} \textsc{P. Baroni, M. Colombo, G.
Mingione:}{\ Regularity for general functionals with double phase,}\emph{\ }%
Calc. Var. Partial Differ. Eq., \textbf{57} (2018), 48 pp.

\bibitem{Bella-Schaffner 2020} \textsc{P. Bella, M. Sch\"{a}ffner:} On the
regularity of minimizers for scalar integral functionals with $(p,q)-$%
growth, Anal. PDE, \textbf{13} (2020), 2241--2257.

\bibitem{Bildhauer-Fuchs 2005} \textsc{M. Bildhauer, M. Fuchs:} {$%
C^{1,\alpha }$-solutions to non-autonomous anisotropic variational problem}%
s, Calc. Var. Partial Differ. Eq., \textbf{24} (2005), 309-340.

\bibitem{Boegelein-Dacorogna-Duzaar-Marcellini-Scheven 2020} \textsc{V. B%
\"{o}gelein, B. Dacorogna, F. Duzaar, P. Marcellini, C. Scheven:} Integral
convexity and parabolic systems, SIAM Journal on Mathematical Analysis
(SIMA), \textbf{52} (2020), 1489-1525.

\bibitem{Boegelein-Duzaar-Marcellini-Scheven JMPA 2021} \textsc{V. B\"{o}%
gelein, F. Duzaar, P. Marcellini, C. Scheven:} Boundary regularity for
elliptic systems with $p,q-$growth, J. Math. Pures Appl., \textbf{159}
(2022), 250--293.

\bibitem{Byun-Oh 2020} \textsc{Sun-Sig Byun, Jehan Oh:} Regularity results
for generalized double phase functionals, Anal. PDE, \textbf{13} (2020),
1269-1300.

\bibitem{Cellina-Staicu 2018} \textsc{A. Cellina, V. Staicu: }On the higher
differentiability of solutions to a class of variational problems of fast
growth, Calc. Var. Partial Differential Equations, \textbf{57} (2018), 66.

\bibitem{Chlebicka-DeFilippis 2019} \textsc{I. Chlebicka, C. De Filippis:}
Removable sets in non-uniformly elliptic problems, Annali di Matematica Pura
ed Applicata, \textbf{199} (2020), 619-649.

\bibitem{Cianchi 1997} \textsc{A. Cianchi:} Boundedness of solutions to
variational problems under general growth conditions, Communications in
Partial Differential Equations, \textbf{22} (1997), 1629-1646.

\bibitem{Cianchi-Mazya 2014} \textsc{A. Cianchi, V.G. Maz'ya:} Global
boundedness of the gradient for a class of nonlinear elliptic systems, Arch.
Rational Mech. Anal.\emph{,} \textbf{212} (2014), 129-177.

\bibitem{Cianchi-Schaffner 2024} \textsc{A. Cianchi, M. Sch\"{a}ffner:}
Local boundedness of minimizers under unbalanced Orlicz growth conditions,
J. Differential Equations, \textbf{401} (2024), 58-92.

\bibitem{Ciani-Skrypnik-Vespri 2023} \textsc{S. Ciani, I. Skrypnik, V.
Vespri:} On the local behavior of local weak solutions to some singular
anisotropic elliptic equations, Adv. Nonlinear Anal., \textbf{12} (2023),
237-265.

\bibitem{Colombo-Mingione 2015} \textsc{M. Colombo, G. Mingione:} Regularity
for double phase variational problems, Arch. Rational Mech. Anal., \textbf{%
215} (2015), 443-496.

\bibitem{Colombo-Mingione 2015 Bounded minimizers} \textsc{M. Colombo, G.
Mingione:} Bounded minimizers of double phase variational integrals, Arch.
Rational Mech. Anal., \textbf{218} (2015), 219-273.

\bibitem{Cupini-Leonetti-Mascolo 2017} \textsc{G. Cupini, G. Leonetti, E.
Mascolo:} Local Boundedness for minimizers of some polyconvex integrals,
Arch. Rational Mech. Anal., \textbf{224 }(2017), 269-289.

\bibitem{Cupini-Marcellini 2025} \textsc{G. Cupini, P. Marcellini:}{\ }
Global boundedness of weak solutions to a class of nonuniformly elliptic
equations, Mathematische Annalen, \textbf{392} (2025), 1519-1539. \
https://doi.org/10.1007/s00208-025-03126-5

\bibitem{Cupini-Marcellini-Mascolo 2014} \textsc{G. Cupini, P. Marcellini,
E. Mascolo:} {Existence and regularity for elliptic equations under $p,q-$%
growth,} {\ Adv. Differential Equations},\text{\ }\textbf{19}{\ (2014),
693-724.}

\bibitem{Cupini-Marcellini-Mascolo 2023} \textsc{G. Cupini, P. Marcellini,
E. Mascolo:} Local boundedness of weak solutions to elliptic equations with $%
p,q-$growth, Math. Eng., \textbf{5} (2023), Paper No. 065, 28 pp.

\bibitem{Cupini-Marcellini-Mascolo regularity 2024} \textsc{G. Cupini, P.
Marcellini, E. Mascolo:} Regularity for nonuniformly elliptic equations with 
$p,q-$growth and explicit $x,u-$dependence, Arch. Rational Mech. Anal. 
\textbf{248} (2024), Paper No. 60, 45 pp.

\bibitem{Cupini-Marcellini-Mascolo Leray-Lions 2024} \textsc{G. Cupini, P.
Marcellini, E. Mascolo:} The Leray-Lions existence theorem under general
growth conditions, J. Differential Equations, \textbf{416} (2025), 1405-1428.

\bibitem{Cupini-Marcellini-Mascolo-Passarelli 2023} \textsc{G. Cupini, P.
Marcellini, E. Mascolo, A. Passarelli di Napoli:} Lipschitz regularity for
degenerate elliptic integrals with $p,q-$growth, Advances in Calculus of
Variations, \textbf{16} (2023), 443-465.

\bibitem{De Filippis JMPA 2022} \textsc{C. De Filippis:} Quasiconvexity and
partial regularity via nonlinear potentials, J. Math. Pures Appl., \textbf{%
163} (2022), 11-82.

\bibitem{DeFilippis-Koch-Kristensen  2024} \textsc{C. De Filippis, C. Koch,
J. Kristensen:} Quantified Legendreness and the Regularity of Minima, Arch.
Rational Mech. Anal., \textbf{248} (2024), Paper No. 69, 51 pp.

\bibitem{DeFilippis-Mingione ARMA 2021} \textsc{C. De Filippis, G. Mingione:}
Lipschitz bounds and nonautonomous integrals, Arch. Rational Mech. Anal., 
\textbf{242} (2021), 973-1057.

\bibitem{DeFilippis-Mingione Inv 2023} \textsc{C. De Filippis, G. Mingione:}
Nonuniformly elliptic Schauder theory, Invent. Math., \textbf{234} (2023),
1109-1196.

\bibitem{DeFilippis-Mingione Notices 2025} \textsc{C. De Filippis, G.
Mingione:} Nonuniform ellipticity in variational problems and regularity,
Notices of the American Mathematical Society, October 2025, 989-999. doi:
https://doi.org/10.1090/noti3218

\bibitem{De Filippis-Stroffolini 2023} \textsc{C. De Filippis, B.
Stroffolini:} Singular multiple integrals and nonlinear potentials, Journal
of Functional Analysis, \textbf{285} (2023), 109952.

\bibitem{DeFilippis-Leonetti-Marcellini-Mascolo 2023} \textsc{F. De
Filippis, F. Leonetti, P. Marcellini, E. Mascolo:} The Sobolev class where a
weak solution is a local minimizer, Atti Accad. Naz. Lincei Cl. Sci. Fis.
Mat. Natur., \textbf{34} (2023), 451-463.

\bibitem{DiBenedetto-Gianazza-Vespri 2016} \textsc{E. Di Benedetto, U.
Gianazza, V. Vespri}{: }Remarks on local boundedness and local H\"{o}lder
continuity of local weak solutions to anisotropic $p-$Laplacian type
equations, J. Elliptic Parabol. Equ., \textbf{2} (2016), 157-169.

\bibitem{Diening-Harjulehto-Hasto-Ruzicka 2011} \textsc{L. Diening, P.
Harjulehto, P. Hasto, M. Ruzicka:} Lebesgue and Sobolev Spaces with Variable
Exponents, Lecture Notes in Mathematics, \textbf{2017}, Springer,
Heidelberg, 2011.

\bibitem{DiMarco-Marcellini 2020} \textsc{T. Di Marco, P. Marcellini:}
A-priori gradient bound for elliptic systems under either slow or fast
growth conditions, Calc. Var. Partial Differential Equations, \textbf{59}
(2020), Paper No. 120, 26 pp.

\bibitem{Egnell  1998} \textsc{H. Egnell:} Existence and non--existence
results for $m-$Laplace equations involving critical Sobolev exponents,
Arch. Ration. Mech. Anal., \textbf{104} (1998), 57-77.

\bibitem{Eleuteri-Marcellini-Mascolo 2016} \textsc{M. Eleuteri, P.
Marcellini, E. Mascolo:} Lipschitz estimates for systems with ellipticity
conditions at infinity, Annali di Matematica, \textbf{195} (2016), 1575-1603.

\bibitem{Eleuteri-Marcellini-Mascolo 2020} \textsc{M. Eleuteri, P.
Marcellini, E. Mascolo:} Regularity for scalar integrals without structure
conditions, Advances in Calculus of Variations, \textbf{13} (2020), 279-300.

\bibitem{Eleuteri-Marcellini-Mascolo-Perrotta 2022} \textsc{M. Eleuteri, P.
Marcellini, E. Mascolo, S. Perrotta:} Local Lipschitz continuity for energy
integrals with slow growth, Ann. Mat. Pura Appl., \textbf{201} (2022),
1005-1032.

\bibitem{Esposito-Leonetti-Mingione 2004} \textsc{L. Esposito, F. Leonetti,
G. Mingione:} Sharp regularity for functionals with $\left( p,q\right) $
growth, J. Differential Equations, \textbf{204} (2004), 5-55.

\bibitem{Gasinski-Winkert 2020} \textsc{L. Gasi\~{n}ski, P. Winkert:}
Constant sign solutions for double phase problems with superlinear
nonlinearity, Nonlinear Analysis, \textbf{195} (2020), 111739

\bibitem{Giga-Tsubouchi 2022} \textsc{Y. Giga, S. Tsubouchi:} Continuity of
derivatives of a convex solution to a perturbed one-Laplace equation by $p-$%
Laplacian, Arch. Ration. Mech. Anal., \textbf{244} (2022), 253-292.

\bibitem{Gilbarg-Trudinger 1977} \textsc{D. Gilbarg, N.S. Trudinger:}
Elliptic partial differential equations of second order, \textbf{224}, no.
2. Berlin, Springer, 1977.

\bibitem{Giusti 2003 book} \textsc{E. Giusti:} Direct methods in the
calculus of variations, World Scientific Publishing Co., Inc., River Edge,
NJ, 2003.

\bibitem{Gmeineder-Kristensen 2022} \textsc{F. Gmeineder, J. Kristensen:}{\ }%
Quasiconvex functionals of $(p,q)-$growth and the partial regularity of
relaxed minimizers, Arch. Rational Mech. Anal., \textbf{248} (2024), no. 5,
Paper No. 80, 125 pp.

\bibitem{Guedda-Veron 1989} \textsc{M. Guedda, L. Veron:} Quasilinear
elliptic equations involving critical Sobolev exponents, Nonlinear Anal., 
\textbf{13} (1989), 879-902.

\bibitem{Hasto-Ok 2022} \textsc{P. H\"{a}st\"{o}, J. Ok:}{\ }Maximal
regularity for local minimizers of non-autonomous functionals, J. Eur. Math.
Soc., \textbf{24} (2022), 1285-1334.

\bibitem{Hirsch-Schaffner 2021} \textsc{J. Hirsch, M. Sch\"{a}ffner:} Growth
conditions and regularity, an optimal local boundedness result, Commun.
Contemp. Math., \textbf{23} (2021), Paper No. 2050029, 17 pp.

\bibitem{Ho-Kim-Winkert-Zhang 2022} \textsc{K. Ho, Y. Kim, P. Winkert, C.
Zhang:} The boundedness and H\"{o}lder continuity of weak solutions to
elliptic equations involving variable exponents and critical growth, J.
Differential Equations, \textbf{313} (2022), 503-532.

\bibitem{Ho-Winkert 2023} \textsc{K. Ho, P. Winkert:} New embedding results
for double phase problems with variable exponents and a priori bounds for
corresponding generalized double phase problems, Calc. Var. Partial
Differential Equations, \textbf{62} (2023), no. 8, Paper No. 227, 38 pp.

\bibitem{Ladyzhenskaya-Uraltseva 1968} \textsc{O. Ladyzhenskaya, N.
Ural'tseva:}{\ }Linear and quasilinear elliptic equations. Translated from
the Russian, New York-London, 1968.

\bibitem{Liskevich-Skrypnik-2009} \textsc{V. Liskevich, I. Skrypnik:} H\"{o}%
lder continuity of solutions to an anisotropic elliptic equation, Nonlinear
Analysis, \textbf{71} (2009), 1699-1708.

\bibitem{Marcellini 1985} \textsc{P. Marcellini:} Approximation of
quasiconvex functions, and lower semicontinuity of multiple integrals,
Manuscripta Mathematica, \textbf{51} (1985), 1-28.

\bibitem{Marcellini 1989} \textsc{P. Marcellini:}{\ Regularity of minimizers
of integrals in the calculus of variations with non standard growth
conditions},\ Arch. Rational Mech. Anal.,\textbf{\ 105}{\ }(1989), 267-284.

\bibitem{Marcellini 1991} \textsc{P. Marcellini:}{\ Regularity and existence
of solutions of elliptic equations with }$p,q-$growth conditions,\ J.
Differential Equations,\textbf{\ 90}{\ (1991), 1-30}.

\bibitem{Marcellini 1993} \textsc{P. Marcellini:}{\ Regularity for elliptic
equations with general growth conditions},\ J. Differential Equations,%
\textbf{\ 105}{\ (1993), 296-333}.

\bibitem{Marcellini JOTA 1996} \textsc{P. Marcellini:} Regularity for some
scalar variational problems under general growth conditions, J. Optim.
Theory Appl., \textbf{90} (1996), 161-181.

\bibitem{Marcellini 2023} \textsc{P. Marcellini:} Local Lipschitz continuity
for $p,q-$PDEs with explicit $u-$dependence, Nonlinear Anal., \textbf{226}
(2023), Paper No. 113066, 26 pp.

\bibitem{Marcellini 2025} \textsc{P. Marcellini:} Some remarks on convexity,
J. Convex Analysis, \textbf{32} (2025), No. 2, 565-577.

\bibitem{Marcellini-Nastasi-Pacchiano Camacho 2025} \textsc{P. Marcellini,
A. Nastasi, C. Pacchiano Camacho:} Unified a-priori estimates for minimizers
under $p,q-$growth and exponential growth, Nonlinear Analysis, \textbf{264}
(2026), art.no. 113982. \ \ https://doi.org/10.1016/j.na.2025.113982

\bibitem{Marino-Winkert 2019} \textsc{G. Marino, P. Winkert:} Moser
iteration applied to elliptic equations with critical growth on the
boundary, Nonlinear Analysis, \textbf{180} (2019), 154-169.

\bibitem{Mihailescu-Pucci-Radulescu 2008} \textsc{M. Mihailescu, P. Pucci,
V. Radulescu:} Eigenvalue problems for anisotropic quasilinear elliptic
equations with variable exponent, J. Math. Anal. Appl., \textbf{340} (2008),
687-698.

\bibitem{Nastasi-Pacchiano Camacho 2023} \textsc{A. Nastasi, C. Pacchiano
Camacho:} Higher integrability and stability of $(p,q)-$quasiminimizers,\ J.
Differential Equations,\textbf{\ 342}{\ (2023), 121-149}.

\bibitem{Ok 2017} \textsc{J. Ok:} Regularity of $\omega -$minimizers for a
class of functionals with non-standard growth, Calc. Var., \textbf{56}
(2017), article number 48.

\bibitem{Papageorgiou-Radulescu 2024} \textsc{N.S. Papageorgiou, V.D. R\u{a}%
dulescu:} Some useful tools in the study of nonlinear elliptic problems,
Expositiones Mathematicae, \textbf{42} (2024), 125616.

\bibitem{Papageorgiou-Radulescu-Zhang 2022} \textsc{N.S. Papageorgiou, V.D. R%
\u{a}dulescu, Y. Zhang:} Resonant double phase equations, Nonlinear
Analysis: Real World Applications, \textbf{64} (2022), 103454, 20 pp.

\bibitem{Perera-Squassina 2018} \textsc{K. Perera, M. Squassina:} Existence
results for double-phase problems via Morse theory, Communications in
Contemporary Mathematics, \textbf{20} (2018), 14 pp.

\bibitem{Pucci-Serrin 2007} \textsc{P. Pucci, J. Serrin:} The maximum
principle, Progress in Nonlinear Differential Equations and Their
Applications, \textbf{73}, Birkh\"{a}user, Switzerland, 2007, p. x + 235 pp.

\bibitem{Radulescu-Stapenhorst-Winkert 2025} \textsc{V.D. R\u{a}dulescu, M.
Stapenhorst, P. Winkert:} Multiplicity results for logarithmic double phase
problems via Morse theory, Bulletin of the London Mathematical Society,
2025, 1-24. \ \ doi: 10.1112/blms.70190

\bibitem{Serrin 1964} \textsc{J. Serrin:} Local behavior of solutions of
quasi-linear elliptic equations, Acta Math., \textbf{111} (1964), 247-302.

\bibitem{Stampacchia 1965} \textsc{G. Stampacchia:} Le probl\`{e}me de
Dirichlet pour les \'{e}quations elliptiques du second ordre \`{a}
coefficients discontinus, Annales de l'institut Fourier, \textbf{15} (1965),
189-257.

\bibitem{Talenti 1979} \textsc{G. Talenti:} Nonlinear elliptic equations,
rearrangements of functions and Orlicz spaces, Ann. Mat. Pura Appl., \textbf{%
120} (1979), 160-184.

\bibitem{Winkert-Zacher 2012} \textsc{P. Winkert, R. Zacher:} A priori
bounds for weak solutions to elliptic equations with nonstandard growth,
Discrete Contin. Dyn. Syst. Ser. S, \textbf{5} (2012), 865-878. \textit{%
Corrigendum:} Discrete Contin. Dyn. Syst. Ser. S, published on-line, 2015,
1-3.

\bibitem{Zhang-Radulescu 2018} \textsc{Qihu Zhang, V.D. R\u{a}dulescu:}
Double phase anisotropic variational problems and combined effects of
reaction and absorption terms, J. Math. Pures Appl., \textbf{118} (2018),
159-203.
\end{thebibliography}
\end{document}